\documentclass[journal]{IEEEtran}

\IEEEoverridecommandlockouts                              

\usepackage{changes}
\usepackage{authblk}
\usepackage{amsfonts}
\usepackage{amsthm}
\usepackage{amsmath,amssymb}
\usepackage{bm}
\usepackage{hyperref,cite}
\usepackage{cleveref}
\usepackage{color}
\usepackage{algorithm}
\usepackage{algorithmic}
\let\labelindent\relax
\usepackage{enumitem}
\usepackage{float}
\usepackage{graphicx}
\usepackage{subfigure}
\usepackage{booktabs}
\usepackage{threeparttable}
\usepackage{hhline}
\usepackage{scalerel}
\usepackage{ulem}
\usepackage{tikz}

\newtheorem{theorem}{Theorem}

\newtheorem{remark}{Remark}
\newtheorem{lemma}{Lemma}
\newtheorem{corollary}{Corollary}
\newtheorem{definition}{Definition}
 
\newcommand{\diag}{\mathop{\rm diag}\nolimits}

\newcommand{\rr}{{\mathbb R}}
\newcommand{\ba}[1]{\begin{array}{#1}}
\newcommand{\ea}{\end{array}}

\newcommand*{\pdot}{\mathbin{\scalerel*{\boldsymbol\odot}{\circ}}}

\begin{document}
\title{\LARGE \bf
A Quadratic Programming Algorithm with $O(n^3)$ Time Complexity}
\author{Liang Wu$^{1}$
and Richard D. Braatz$^{1}$%
\thanks{$^{1}$Massachusetts Institute of Technology, Cambridge, MA 02139, USA, {\tt\small \{liangwu,braatz\}@mit.edu}.
}
}
\maketitle

\begin{abstract}
Solving linear systems and quadratic programming (QP) problems are both ubiquitous tasks in the engineering and computing fields. Direct methods for solving linear systems, such as Cholesky, LU, and QR factorizations, exhibit \textit{data-independent} time complexity of $O(n^3)$. This raises a natural question: could there exist algorithms for solving QPs that also achieve \textit{data-independent} time complexity of $O(n^3)$? This is critical for offering an execution time certificate for real-time optimization-based applications such as model predictive control. This article first demonstrates that solving real-time strictly convex QPs, Lasso problems, and support vector machine problems can be turned into solving box-constrained QPs (Box-QPs), which support a cost-free initialization strategy for feasible interior-point methods (IPMs). Next, focusing on solving Box-QPs, this article replaces the \textit{exact Newton step} with an \textit{approximated Newton step} (substituting the matrix-inversion operation with multiple rank-1 updates) within feasible IPMs. For the first time, this article proposes an \textit{implementable} feasible IPM algorithm with $O(n^3)$ time complexity, by proving the number of iterations is exact $O(\sqrt{n})$ and the number of rank-1 updates is bounded by $O(n)$. Numerical validations/applications and codes are provided.
\end{abstract}


\section{Introduction}
Similar to solving linear systems, quadratic programming (QP) is ubiquitous and widely applied across diverse fields such as machine learning, finance, operations, energy, transportation, computer vision, signal processing, bioinformatics, robotics, and control. In particular, QP-based control approaches such as model predictive control (MPC) \cite{borrelli2017predictive,wu2023simple}, moving horizon estimation \cite{kuhl2011real}, CLF-CBF based QPs \cite{ames2016control}, QP-based minimum snap trajectory generation \cite{mellinger2011minimum} for quadrotor maneuvering, and QP-based task-space control paradigms \cite{escande2014hierarchical,bouyarmane2018quadratic} for robotics, are classified \textit{online}, since they must be executed in \textit{real-time}, in contrast to the \textit{offline} QPs, which have no such requirement. 

In such real-time scenarios, QP algorithms are required to provide an execution time certificate (on time-predictable computer architectures \cite{schoeberl2009time}), ensuring that a solution can be returned within a specified sampling time.  An execution time certificate was recently provided in the analog optimization paradigm   \cite{wu2025arbitrarily}. Still, finding such a certificate remains an open problem in the mainstream numerical optimization paradigm, as for numerically \textit{iterative} optimization algorithms; certifying the number of iterations and the number of operations needed to solve real-time QPs with time-varying Hessian matrix data is inherently difficult \cite{richter2011computational, giselsson2012execution, patrinos2013accelerated, cimini2017exact, arnstrom2019exact, cimini2019complexity,arnstrom2020complexity, arnstrom2021unifying, okawa2021linear}.  References \cite{richter2011computational,giselsson2012execution} and \cite{cimini2017exact, cimini2019complexity, arnstrom2019exact, arnstrom2020complexity, arnstrom2021unifying, okawa2021linear} adopt first-order methods and active-set based methods to provide worst-case iteration complexity analysis, respectively, which are dependent on the Hessian matrix of QPs.

An execution time certificate is available when solving linear systems numerically, as there exist \textit{direct} methods such as Cholesky factorization, LU factorization, and QR factorization. Those factorization methods can be viewed as \textit{iterative} algorithms in which the number of iterations is \textit{exactly} 
the problem dimension $n$ (\textit{data-independent}) with certified $O(n^3)$ time complexity. 
In this perspective, our previous box-constrained QP (Box-QP) algorithm \cite{wu2023direct}, whose iteration complexity is the \textit{exact} and \textit{data-independent} (or \textit{only dimension-dependent}) value of $
\!\left\lceil\frac{\log(\frac{2n}{\epsilon})}{-2\log(\frac{\sqrt{2n}}{\sqrt{2n}+\sqrt{2}-1})}\right\rceil \! + 1
$ (where $n$ and $\epsilon$ denote the problem dimension and the predefined optimality level, respectively), rather than an upper bound as provided in previous works, can be regarded as being \textit{direct}. The \textit{direct} Box-QP algorithm is then used to offer an execution time certificate for real-time MPC applications \cite{wu2024time,wu2024execution} and extended to handle general QPs via the $\ell_1$-penalty soft-constrained scheme \cite{wu2024parallel}. The \textit{direct} Box-QP algorithm is based on feasible interior-point methods (IPMs), which offer certified $O(\sqrt{n})$ iteration complexity (better than the usual $O(n)$ iteration complexity of infeasible IPMs \cite{roos2006full}). Feasible and infeasible IPMs differ in whether the initial point is feasible or not, and Ref.\ \cite{wu2023direct} exploited the structure of Box-QP to propose a cost-free initialization strategy. However, solving general convex QPs (including linear programming, LP), which are possibly infeasible, requires detecting infeasibility as well with no known way of finding a strictly feasible point in a cost-free manner. 
To address this issue, our previous work \cite{wu2025eiqp} proposed a novel infeasible IPM and proved its iteration complexity to be $O(\sqrt{n})$, with the \textit{exact} and \textit{data-independent} value $\!\left\lceil\frac{\log(\frac{n+1}{\epsilon})}{-\log(1-\frac{0.414213}{\sqrt{n+1}})}\right\rceil$.

Although the algorithms of \cite{wu2023direct} and \cite{wu2025eiqp} for solving QPs can be regarded as being \textit{direct} (thanks to their \textit{exact} and \textit{data-independent} iteration complexity, thus analogous to \textit{direct} methods in solving linear systems), their time complexity is $O(n^{3.5})$ (resulting from $O(\sqrt{n})$ iteration complexity multiplied by the $O(n^3)$ time complexity of solving a linear system in each iteration of IPMs), which is larger than the time complexity of \textit{direct} methods in solving linear systems. Moreover, there is a well-known $O(\log(n))$ iteration complexity conjecture in the field of IPM algorithms, which has remained unproven for over four decades. This conjecture stems from the empirical observation that heuristic IPM algorithms typically exhibit $O(\log(n))$ iteration complexity in practice. If the $O(\log(n))$ conjecture were proved, the time complexity of IPM algorithms would be $O(\log(n)\times n^3)$, which is close to the $O(n^3)$ iteration complexity of \textit{direct} methods in solving linear systems.

This article proposes a \textit{direct} IPM-based QP algorithm with a certified \textit{exact} and \textit{data-independent} number of iterations, while achieving the same order of time complexity, $O(n^3)$, as \textit{direct} methods in solving linear systems.

\subsection{Related work}
The development of $O(n^3)$ IPM algorithms dates back to the 1980s. The core idea behind these algorithms is that, while the number of iterations remains at the proven $O(\sqrt{n})$ level, the costly matrix inversion at each iteration is replaced by an approximation using multiple rank-1 updates, a technique originally introduced in \cite{karmarkar1984new}. The first major breakthrough came from \cite{vaidya1987algorithm} and \cite{gonzaga1989algorithm}, who independently proved that the total number of such rank-1 updates is bounded by $O(n)$, resulting in an overall $O(n^3L)$ time complexity (where $L$ is bounded by the number of bits in the input) for solving LP. Other $O(n^3L)$ IPM-based algorithms for solving LP were proposed in \cite{roos1989,ye1991n,mizuno1991nrhol,anstreicher1992long,bosch1993mizuno,wu1994rank,anstreicher1999linear}. The same technique, which uses multiple rank-1 updates to approximate the matrix inversion within the IPM framework, was further extended to solve convex QPs \cite{monteiro1989interiorII,goldfarb1990n,tiande1996properties} and linear complementary problems \cite{kojima1989polynomial,mizuno1990n,mizuno1991n,mizuno1992new,mcshane1996superlinear} to obtain $O(n^3L)$ time complexity, respectively. 

However, all the aforementioned $O(n^3L)$ IPM-based algorithms are not practically \textit{implementable}, and none of the works \cite{roos1989,ye1991n,mizuno1991nrhol,anstreicher1992long,bosch1993mizuno,wu1994rank,anstreicher1999linear,monteiro1989interiorII,goldfarb1990n,tiande1996properties,kojima1989polynomial,mizuno1990n,mizuno1991n,mizuno1992new,mcshane1996superlinear} provide numerical experiments. This is because all of them are based on feasible IPM frameworks and assume that a strictly feasible point is given. Finding a strictly feasible point in feasible IPM frameworks generally requires solving another LP. Moreover, none of their time complexity results is \textit{data-independent} (the results explicitly depend on $L$), and therefore can not provide an execution time certificate for real-time optimization-based applications.

\subsection{Contributions}
This article begins by addressing a Box-QP, as it supports finding a strictly feasible initial point for feasible IPMs as demonstrated in our previous work \cite{wu2023direct}. It then demonstrates how to transform a general strictly convex QP, a Lasso problem, and a support vector machine (SVM) (or regression) into a Box-QP, highlighting the broad applicability. The novel contributions of this article are to
\begin{itemize}
    \item[1)] develop a unified and simple proof framework for feasible IPM algorithms with \textit{exact Newton step} and \textit{approximated Newton step}, respectively resulting in $O(n^{3.5})$ and $O(n^3)$ time complexity;
    \item[2)] prove that the proposed $O(n^3)$-time-complexity algorithm has the \textit{exact} and \textit{data-independent} number of iterations   
    \[
\mathcal{N}_\mathrm{iter}= \!\left\lceil\frac{\log\!\left(\frac{2n+\alpha\sqrt{2n}}{\epsilon}\right)}{-\log\!\left(1-\frac{\beta}{\sqrt{2n}}\right)}\right\rceil
    \]
    and the \textit{data-independent} total number of rank-1 updates bounded by
    \[
    \mathcal{N}_{\mathrm{rank}-1}\leq \! \left\lceil\frac{4\eta(\mathcal{N}_\mathrm{iter}-1)\sqrt{n}}{\left(1-\eta\right)\log(1+\delta)}\right\rceil.
    \]
     (where $\alpha,\beta,\eta,\delta$ are \textit{data-independent} constants), thereby being able to offer an execution time certificate for real-time optimization-based applications;
    \item[3)] show that the proposed $O(n^3)$-time-complexity algorithm is \textit{simple to implement} and
    \textit{matrix-free} in the iterative procedures.
\end{itemize}

To the best of the author's knowledge, this article is the first to propose an \textit{implementable} feasible IPM algorithm with $O(n^3)$ time complexity in solving QPs.

\subsection{Notation}
$\mathbb{R}^n$ denotes the space of $n$-dimensional real vectors, $\mathbb{R}^n_{++}$ is the set of all positive vectors of $\mathbb{R}^n$, and $\mathbb{N}_+$ is the set of positive integers. Given a vector $x$ and a matrix $X$, $x_i$ denotes the $i$th element of $x$, $X_{i,i}$ denotes the element in the $i$th row and $i$th column of $X$, and $X_{\cdot,i}$ denotes the $i$th column of $X$. The vector of all ones is denoted by $e=(1,\cdots{},1)^\top$. Given a set $I$, the notation $|I|$ denotes the number of elements in set $I$. $\left\lceil x\right\rceil$ maps $x$ to the least integer greater than or equal to $x$. For $z,y\in\mathbb{R}^n$, let $\mathrm{col}(z,y)=[z^{\top},y^{\top}]^{\top}$. Given two arbitrary vectors $z,y \in\mathbb{R}^n_{++}$, their Hadamard product is $z\pdot y = (z_1y_1,z_2y_2,\cdots{},z_ny_n)^\top$, $\frac{z}{y}=\big(\frac{z_1}{y_1},\frac{z_2}{y_2},\cdots{},\frac{z_n}{y_n}
\big)^{\!\top}$. 

\section{Problem Formulations}\label{sec_BoxQP}
This article starts from a convex Box-QP,
\begin{equation}\label{eqn_box_QP}
    \begin{aligned}
        \min_{z\in\rr^{n}}&~ \frac{1}{2}z^\top H z + z^\top h\\
        \mathrm{s.t.}&~ -e \leq z \leq e
    \end{aligned}
\end{equation}
where $H=H^\top\in\rr^{n\times n}\succeq0$. The general box constraint $l\leq z\leq u$ ($l<u$) can be easily transformed into the scaled box constraint $-e\leq z\leq e$. In fact, many practical optimization problems, such as general strictly convex QPs, non-smooth Lasso problems \cite{tibshirani1996regression}, and support vector machine (SVM) \cite{cortes1995support} (or regression \cite{smola2004tutorial}) (see Appendix \ref{sec_SVM_BoxQP}), can be transformed into Box-QP. 

Transforming those problems into Box-QPs is the key to deriving our proposed $O(n^3)$ QP algorithm, because the special structure of Box-QP \eqref{eqn_box_QP} has no cost in finding a strictly feasible initial point for feasible IPMs, as shown in our prior work \cite{wu2023direct}. 

\subsection{Strictly convex QP to Box-QP via $\ell_1$ penalty}\label{sec_ell_1_penalty}
Given a standard strictly convex QP,
\begin{equation}\label{eqn_standard_QP}
    \begin{aligned}
\min_{y\in\rr^{n_y}}&~ \frac{1}{2}y^\top Q y + y^\top q\\
        \mathrm{s.t.}&~ Gy\leq g
    \end{aligned}
\end{equation}
where $Q\in\rr^{n_y\times n_y}\succ0$, $q\in\rr^{n_y}$, $G\in\rr^{n_g\times n_y}$, and $g\in\rr^{n_g}$. QP \eqref{eqn_standard_QP} often arises in real-time optimization-based applications, such as Model Predictive Control (MPC). In MPC scenarios, the resulting QP \eqref{eqn_standard_QP} can be infeasible, that is, the set $\{y\in\rr^{n_y}: Gy\leq g\}$ is empty. A common practical approach is to use soft-constrained formulations. Our previous work \cite{wu2024parallel} uses the exact $\ell_1$-penalty soft-constraint scheme for QP \eqref{eqn_standard_QP},
\begin{equation}\label{eqn_l1_peantly_QP}
        \min_y~\frac{1}{2}y^\top Q y + y^\top q + \|\rho\pdot\max (0,Gy-g)\|_1
\end{equation}
(where $\rho\in\rr_{++}^{n_g}$ denotes the penalty weight vector) to equivalently transform \eqref{eqn_standard_QP} into Box-QP \eqref{eqn_box_QP}.
\begin{lemma}(see \cite[Col.\ 2]{wu2024parallel})
    The $\ell_1$-penalty soft-constrained QP \eqref{eqn_l1_peantly_QP} has a unique solution $y^*$, which can be recovered by the optimal solution $z^*$ of Box-QP \eqref{eqn_box_QP} with
    \begin{equation}
        y^* = -Q^{-1}\!\left(q+\frac{1}{2} G^\top\!\left(\rho\pdot z^* + \rho\right)\right)
    \end{equation}
    if the data of Box-QP \eqref{eqn_box_QP} are assigned as 
    \[
    \begin{aligned}
    &H\leftarrow \diag(\rho)GQ^{-1}G^\top\diag(\rho),\\
    &h\leftarrow \diag(\rho)\left(GQ^{-1}G^\top\rho+2(GQ^{-1}q+g)\right).
    \end{aligned}
    \]
\end{lemma}

\subsection{Lasso to Box-QP}\label{sec_Lasso_BoxQP}
Given the Lasso ($\ell_1$-regularized least squares) problem,
\begin{equation}\label{problem_lasso} 
    \min_{x\in\mathbb{R}^n}~\frac{1}{2}\|Ax-b\|_2^2 + \lambda \|x\|_1
\end{equation}
where $A\in\mathbb{R}^{m\times n}$ and $b\in\mathbb{R}^m$, assume that $m\geq n$ and $A^\top A$ has full rank.

To build the connection with Box-QP \eqref{eqn_box_QP},  first
reformulate Lasso \eqref{problem_lasso} by introducing an auxiliary variable $y=x$ as
\begin{equation}\label{problem_add_y_equality_constrained}
    \begin{aligned}
        \min_{x,y}&~ \frac{1}{2}\|Ax-b\|_2^2 + \lambda \|y\|_1\\
        \text{s.t.}&~ x - y=0
    \end{aligned}
\end{equation} 
in which strong duality holds. Because its objective function is convex and its affine equality constraint satisfies Salter's condition, see \cite[Sec. 5.2.3]{boyd2004convex}. This strong duality of Problem (\ref{problem_add_y_equality_constrained}) allows us to solve \eqref{problem_add_y_equality_constrained} via its dual problem and transform Lasso \eqref{problem_lasso} to Box-QP.
\begin{lemma}
The optimal solution $x^*$ of Lasso \eqref{problem_lasso} can be recovered by the optimal solution $z^*$ of Box-QP \eqref{eqn_box_QP} with
\begin{equation}\label{eqn_recover_primal}
    x^*=\left(A^\top A\right)^{\!-1}\!\left(A^\top b-\lambda z^*\right),
\end{equation}
if the data of Box-QP \eqref{eqn_box_QP} are assigned as 
\[
H=\lambda\!\left(A^\top A\right)^{\!-1}, \quad h=-\!\left(A^\top A\right)^{\!-1} \! A^\top b.
\]
\end{lemma}

\begin{proof}
For Problem \eqref{problem_add_y_equality_constrained}, introduce the Lagrange multiplier $z\in\mathbb{R}^n$, for which the corresponding Lagrange function is
\[
\mathcal{L}(x,y,z)=\frac{1}{2}\|Ax-b\|_2^2+\lambda \|y\|_1+z^\top(x-y).
\]
By minimizing $\mathcal{L}(x,y,z)$ with respect to primal variables $x,y$, the dual function is
\[
\mathcal{D}(z)=\min_{x}\left\{\frac{1}{2}\|Ax-b\|_2^2+z^\top x\right\} -\max_{y}\left\{z^\top y-\lambda\|y\|_1\right\}.
\]
The first term of $\mathcal{D}(z)$ has the closed-form solution
\[
x^*=\left(A^\top A\right)^{\!-1}\!\left(A^\top b-z^*\right),
\]
which shows that the optimal solution $x^*$ can be recovered by $z^*$. Based on the dual norm (see \cite{boyd2004convex}), the second term of $\mathcal{D}(z)$ satisfies
\[
\max_y\left\{z^\top y-\lambda\|y\|_1\right\}=\left\{
\begin{array}{@{}r}
     0, \quad \|z\|_\infty \leq \lambda, \\
     +\infty, \quad\|z\|_\infty>\lambda.
\end{array}
\right.
\]
Thus, obtaining $z^*$ is to maximize  $\mathcal{D}(z)$, which is equivalent to the Box-QP:
\[
\begin{aligned}
    \max_z&~\frac{1}{2} \!\left[\left(A^\top A\right)^{\!-1}\!\left(A^\top b-z\right)\right]^\top(A^\top A)\left(A^\top A\right)^{\!-1}\!\left(A^\top b-z\right)\\
    &+z^\top \left(A^\top A\right)^{\!-1}\!\left(A^\top b-z\right)\\
    \text{s.t.}&~ -\lambda\leq z \leq \lambda,\\
    \Leftrightarrow\quad&\\
    \min_z&~ \frac{1}{2}z^\top (A^\top A)^{-1} z+z^\top (-(A^\top A)^{-1}A^\top b)\\
    \text{s.t.}&~ -\lambda\leq z \leq \lambda.
\end{aligned}
\]
which can easily be scaled to Box-QP \eqref{eqn_box_QP}. This completes the proof.
\end{proof}

\section{Preliminary: $O(n^{3.5}$) Feasible IPM Algorithm}
According to \cite[Ch. 5]{boyd2004convex}, the KKT condition of the Box-QP \eqref{eqn_box_QP} is the nonlinear equations
\begin{subequations}\label{eqn_KKT}
\begin{align}
    Hz + h + \gamma - \theta = 0,\label{eqn_KKT_a}\\
    z + \phi - e=0,\label{eqn_KKT_b}\\
    z - \psi + e=0,\label{eqn_KKT_c}\\
    \gamma \pdot\phi = 0,\label{eqn_KKT_d}\\
    \theta \pdot\psi = 0,\label{eqn_KKT_e}\\
(\gamma,\theta,\phi,\psi)\geq0,\label{eqn_KKT_f}
\end{align}
\end{subequations}
where $\gamma\in\rr^{n}$ and $\theta\in\rr^{n}$ denote the Lagrangian variables of the lower bound and upper bound, respectively, and $\phi\in\rr^{n}$ and $\psi\in\rr^{n}$ denote the slack variables of the lower bound and upper bound, respectively.

In solving the nonlinear KKT Eq.\  \eqref{eqn_KKT}, primal-dual IPMs, akin to Newton's method, involve linearizing KKT Eq.\ \eqref{eqn_KKT} around the current iterate $(z,\gamma>0,\theta>0,\phi>0,\psi>0)$ and solving a resulted linear system to obtain the moving direction $(\Delta z,\Delta\gamma,\Delta\theta,\Delta\phi,\Delta\psi)$ at each iteration. A step-size $\eta\in(0,1]$ is found by a line search procedure to ensure the next iterate $(z^+,\gamma^+,\theta^+,\phi^+,\psi^+)\leftarrow(z,\gamma,\theta,\phi,\psi)+\eta(\Delta z,\Delta\gamma,\Delta\theta,\Delta\phi,\Delta\psi)$ satisfying $(\gamma^+,\theta^+,\phi^+,\psi^+)>0$. This procedure has the disadvantage that the step size $\eta$ is often very small ($\eta\ll 1$) and thus results in low efficiency.

To address this issue, path-following primal-dual IPM algorithms introduce a positive parameter $\tau$ to replace \eqref{eqn_KKT_d} and \eqref{eqn_KKT_e} with the equations
\begin{subequations}\label{eqn_KKT_tau}
\begin{align}
    \gamma \phi = \tau e,\label{eqn_KKT_tau_d}\\
    \theta \psi = \tau e.\label{eqn_KKT_tau_e}
\end{align}
\end{subequations}
In path-following primal-dual IPM algorithms, the well-known theory is the existence of {\em central path}: for arbitrary $\tau>0$ there exists one unique solution $(z_{\tau},\gamma_{\tau},\theta_{\tau},\phi_{\tau},\psi_{\tau})$, the path $\tau \rightarrow (z_{\tau},\gamma_{\tau},\theta_{\tau},\phi_{\tau},\psi_{\tau})$ is called the {\em central path} \cite{wright1997primal}, and $(z_{\tau},\gamma_{\tau},\theta_{\tau},\phi_{\tau},\psi_{\tau})$ goes to a solution of \eqref{eqn_KKT} when $\tau$ decreasingly closes to $0$. Let define the primal-dual feasible set $\mathcal{F}$ and strictly feasible set $\mathcal{F}_{\mathrm{int}}$ as
\[
\begin{aligned}
    \mathcal{F}=\{(z,\gamma,\theta,\phi,\psi)|\eqref{eqn_KKT_a}\mathrm{-}\eqref{eqn_KKT_c} ~\text{hold},(\gamma,\theta,\phi,\psi)\geq0\},\\
    \mathcal{F}_{\mathrm{int}}=\{(z,\gamma,\theta,\phi,\psi)|\eqref{eqn_KKT_a}\mathrm{-}\eqref{eqn_KKT_c}~\text{hold},(\gamma,\theta,\phi,\psi)>0\}.
\end{aligned}
\]
According to whether the initial point $(z^0,\gamma^0,\theta^0,\phi^0,\psi^0)\in\mathcal{F}_{\mathrm{int}}$ or $(z^0,\gamma^0,\theta^0,\phi^0,\psi^0)\notin\mathcal{F}_{\mathrm{int}}$, path-following primal-dual IPM algorithms are divided into \emph{feasible} and \emph{infeasible} path-following IPMs, respectively. Generally, \emph{feasible} path-following IPMs have $O(\sqrt{n})$ iteration complexity and \emph{infeasible} path-following IPMs have $O(n)$ iteration complexity, some exceptions exist such as \emph{homogeneous infeasible} path-following IPMs which enjoy $O(\sqrt{n})$ but at the cost of ill-conditioning linear systems to be solved, see our previous work \cite{wu2025eiqp}. Our considered Box-QP \eqref{eqn_box_QP} offers an initial point $(z^0,\gamma^0,\theta^0,\phi^0,\psi^0)\in\mathcal{F}_{\mathrm{int}}$ without any cost, thus the following sections will only focus on \emph{feasible} path-following IPMs.

To simplify the presentation, introduce
\begin{equation}
    x=\mathrm{col}(\gamma,\theta) \in \mathbb{R}^{2n}, ~s=\mathrm{col}(\phi,\psi) \in \mathbb{R}^{2n},
\end{equation}
and define the neighborhood:
\begin{equation}
\mathcal{F}_{\mathrm{cen}}(\alpha,\tau)\triangleq\left\{(z,x,s)\in\mathcal{F}_{\mathrm{int}}:\frac{ \left\|x\pdot s- \tau e\right\|}{\tau}\leq \alpha\right\},  
\end{equation}
where $\alpha>0$, $\tau>0$.

\subsection{Cost-free initialization strategy}
Deriving the $O(\sqrt{n})$ iteration complexity of \emph{feasible} IPM algorithms often requires an initial point $ (z^0,x^0,s^0)\in\mathcal{F}_{\mathrm{cen}}(\alpha,\tau)$. For Box-QP \eqref{eqn_box_QP}, we present a cost-free initialization strategy by using a scaling approach.
\begin{remark}\label{remark_initialization_stragegy}
For $h=0$, the optimal solution of Box-QP \eqref{eqn_box_QP} is $z^*=0$. For $h\neq0$, we first scale its objective as 
\[
\min_z \tfrac{1}{2} z^\top (\tfrac{2\lambda}{\|h\|_\infty}H) z + z^\top (\tfrac{2\lambda}{\|h\|_\infty}h)
\]
which does not affect the optimal solution for any $\lambda>0$. With the definitions
\[
\tilde{H} = \frac{1}{\|h\|_\infty}H, \quad \tilde{h}=\frac{1}{\|h\|_\infty} h,
\]
it follows that $\|\tilde{h}\|_\infty =1$.
Then we replace \eqref{eqn_KKT_a} by 
\[
2\lambda \tilde{H}z+2\lambda\tilde{h}+\gamma-\theta=0,
\]
and the cost-free initialization strategy for Box-QP \eqref{eqn_box_QP} is
\begin{equation}\label{eqn_initialization_stragegy}
\begin{aligned}
    z^0 &= 0,\\
    \gamma^0 &=e - \lambda \tilde{h},\\
    \theta^0 &=e + \lambda \tilde{h},\\
    \phi^0 &= e, \\
    \psi^0 &= e.    
\end{aligned}
\end{equation}
This set of values is in $\mathcal{F}_{\mathrm{int}}$. 

Furthermore, let $\lambda=\frac{\alpha}{\sqrt{2n}}$ and $\tau^0=1$; then
\[
\begin{aligned}
   \frac{\|x^0s^0-\tau^0e\|}{\tau^0} &=  \left\|\left[\begin{array}{c}
        1-\lambda\Tilde{h}  \\
        1+\lambda\Tilde{h} 
        \end{array}\right] - e\right\| = \left\|\left[\begin{array}{c}
        -\lambda\Tilde{h}  \\
        \lambda\Tilde{h} 
        \end{array}\right]\right\|\\
        &=\sqrt{2}\lambda\|\Tilde{h}\|\\
        &\leq\sqrt{2}\lambda\sqrt{n}\|\Tilde{h}\|_\infty= \lambda\sqrt{2n}=\alpha,
\end{aligned}
\]
which shows that the initialization point from \eqref{eqn_initialization_stragegy} is in $\mathcal{F}_{\mathrm{cen}}(\alpha,\tau)$.
\end{remark}

\subsection{$O(n^3)$ cost of each iteration}
The idea of \emph{feasible} IPM algorithms is to continuously decrease $\tau$ to $0$ and generate iterates that always stay in $\mathcal{F}_{\mathrm{cen}}(\alpha,\tau)$ and the duality gap $x^\top s$ approaches zero. Specially, given a current iterate $(z,x,s)\in\mathcal{F}_{\mathrm{cen}}(\alpha,\tau)$, a direction $(\Delta z,\Delta x, \Delta s)$ can be calculated by solving the linear system
\begin{subequations}\label{eqn_KKT_delta}
\begin{align}
   2\lambda\tilde{H}\Delta z +\Omega\Delta x &= 0,\label{eqn_KKT_delta_a}\\
    \Omega^\top\Delta z + \Delta s &=0,\label{eqn_KKT_delta_b}\\
    s\pdot\Delta x+ x\pdot\Delta s &= \tau e- x\pdot s\label{eqn_KKT_delta_c}
\end{align}
\end{subequations}
obtained by linearizing \eqref{eqn_KKT_a}--\eqref{eqn_KKT_c} and \eqref{eqn_KKT_tau}, 
where $\Omega=[I,-I] \in\mathbb{R}^{n \times 2n}$. As $H\succeq0$, one feature of the solution of \eqref{eqn_KKT_delta} is
\begin{equation}\label{eqn_Delta_x_Delta_s_positive}
    \Delta x^\top \Delta s = \Delta x^\top (-\Omega^\top \Delta z)=\Delta z^\top(2\lambda\tilde{H})\Delta z\geq0,
\end{equation}
which is critical in proving the iteration complexity. By letting
\begin{subequations}\label{eqn_Delta_gamma_theta_phi_psi}
    \begin{align}
        &\Delta \gamma=\frac{\gamma}{\phi}\Delta z+\tau\frac{1}{\phi}-\gamma,\\
        &\Delta \theta=-\frac{\theta}{\psi}\Delta z+\tau\frac{1}{\psi}-\theta,\\
        &\Delta\phi = - \Delta z,\\
        &\Delta\psi = \Delta z,
    \end{align}
\end{subequations}
\eqref{eqn_KKT_delta} can be reduced into a more compact system of linear equations,
\begin{equation}{\label{eqn_compact_linsys}}
\Big(2\lambda\tilde{H}+\diag\!\Big(\frac{\gamma}{\phi}\Big) + \diag\!\Big(\frac{\theta}{\psi}\Big) \!\Big) \Delta z=\tau\frac{1}{\psi}-\tau\frac{1}{\phi}+ \gamma - \theta.
\end{equation}
This Newton step takes $O(n^3)$ computation cost by using Cholesky factorization. 

\subsection{$O(\sqrt{n})$ iteration complexity}
Most previous works adopt the update rule for $\tau$: 
\[
\tau^+\leftarrow\left(1-\frac{\beta}{\sqrt{2n}}\right)\!\mu
\]
(where $\mu=\frac{x^\top s}{2n}$ denotes the duality gap measure and $\beta$ denotes the chosen decreasing rate parameter) to derive $O(\sqrt{n})$ iteration complexity. To simplify and unify the proof of $O(n^{3.5})$ and $O(n^3)$ time complexity in feasible IPMs, this paper adopts the update rule for $\tau$ as 
\begin{equation}\label{eqn_assumption_tau}
         \tau^+ = \left(1-\frac{\beta}{\sqrt{2n}}\right)\!\tau.
    \end{equation}
The following Corollary shows how to choose appropriate values for $\alpha$ and $ \beta$ to ensure the positivity of sequences $\{x\}$ and $\{s\}$, as well as the certified $O(\sqrt{n})$ iteration complexity.
\begin{corollary}\label{corollary_n_35}
   Consider any value of $\alpha\in(0,1)$ that satisfies
   \[
   \begin{aligned}
       \mu&=\frac{\alpha}{1-\alpha}\in(0,1),\\
       \beta&=\frac{\alpha-\sigma}{1+\frac{\alpha}{\sqrt{2n}}}>0,
   \end{aligned}
   \]
   where $\sigma=\frac{\alpha^2}{2(1-\alpha)}$ (i.e., $\alpha\in(0,0.5)$).\footnote{For example, choosing $\alpha=0.3$ results in $\mu=0.4286\in(0,1),~\sigma=0.0643<\alpha\Leftrightarrow\beta>0$.}
   
   Suppose that 
    \[
    (x,s)>0, \quad \xi\triangleq\frac{\|x\pdot s-\tau e\|}{\tau}\leq\alpha.
    \]
    Using the update rule \eqref{eqn_assumption_tau} for $\tau$, the next iterate $(z^+,x^+,s^+)$, given by 
\begin{equation}\label{eqn_full_newton_step}
        (z^+, x^+,s^+) = (z, x,s) + (\Delta z, \Delta x, \Delta s),
    \end{equation}
    (where $(\Delta z, \Delta x, \Delta s)$ is from \eqref{eqn_Delta_gamma_theta_phi_psi} and  \eqref{eqn_compact_linsys}), satisfies
    \[
(x^+,s^+)\geq0,\quad \xi^+\triangleq\frac{\|x^+\pdot s^+-\tau^+ e\|}{\tau^+}\leq\alpha.
    \]
     Furthermore, after \textbf{exact}
\begin{equation}\label{eqn_iteration_number}
\mathcal{N}_{iter}=\left\lceil\frac{\log\left(\frac{2n+\alpha\sqrt{2n}}{\epsilon}\right)}{-\log\left(1-\frac{\beta}{\sqrt{2n}}\right)}\right\rceil
    \end{equation}
    iterations, the inequality $x^\top s\leq\epsilon$ holds.
\end{corollary}
\begin{proof}
    See Appendix \ref{sec_corollary_n_35}.
\end{proof} 
\begin{remark}
The number of iterations $\mathcal{N}$ in  \eqref{eqn_iteration_number} is proved to be $O(\sqrt{n})$ (as $-\log(1-\beta/\sqrt{2n})\geq\beta/\sqrt{2n}$ when $0<\beta/\sqrt{2n}<1$) and each iteration requires solving the linear system \eqref{eqn_compact_linsys} whose [flops] cost is $O(n^3)$. Thus, the total time complexity is $O(n^{3.5})$. The pseudo-code of an $O(n^{3.5})$ \emph{feasible} IPM algorithm is shown in Alg.\  \ref{alg_n_3_dot_5}.
\end{remark}
\begin{algorithm}[H]
    \caption{An $O(n^{3.5})$ Feasible IPM Algortihm for Box-QP \eqref{eqn_box_QP}}\label{alg_n_3_dot_5}
    \textbf{Input}: Given the Box-QP data $(H,h)$, the dimension $n$.
    \vspace*{.1cm}\hrule\vspace*{.1cm}
    \textbf{if }$\|h\|_\infty=0$, \textbf{return} $z^*=0$;\\
    \textbf{otherwise},\\
     Choose $0<\alpha<0.5$ and an optimality level $\epsilon$ (such as $\epsilon=10^{-6}$). Set $\lambda=\frac{\alpha}{\sqrt{2n}}$. Cache $2\lambda\tilde{H}=\frac{2\lambda}{\|h\|_\infty}H \mathrm{and} \tilde{h}=\frac{1}{\|h\|_\infty}h$. 
    Initialize $(z,\gamma,\theta,\phi,\psi)$ as per  \eqref{eqn_initialization_stragegy} and $\tau=1$. Compute $\sigma=\frac{\alpha^2}{2(1-\alpha)}$, $\beta=\frac{\alpha-\sigma}{1+\frac{\alpha}{\sqrt{2n}}}$, and the required number of iterations  $\mathcal{N}_\mathrm{iter}=\!\left\lceil\frac{\log\left(\frac{2n+\alpha\sqrt{2n}}{\epsilon}\right)}{-\log\left(1-\frac{\beta}{\sqrt{2n}}\right)}\right\rceil$;\\ \\
    \textbf{for} $k=1,2,\cdots{},\mathcal{N}_\mathrm{iter}$, then stop
    \begin{enumerate}[label*=\arabic*., ref=\theenumi{}]    
       \item Compute $\Delta z$ by  solving  \eqref{eqn_compact_linsys} and $(\Delta\gamma,\Delta\theta,\Delta\phi,\Delta\psi)$ by  \eqref{eqn_Delta_gamma_theta_phi_psi};\\
      \item $(z,\gamma,\theta,\phi,\psi)\leftarrow(z,\gamma,\theta,\phi,\psi) + (\Delta z,\Delta\gamma,\Delta\theta,\Delta\phi,\Delta\psi)$;\\
      \item $\tau\leftarrow\left(1-\frac{\beta}{2\sqrt{n}}\right)\tau$;
    \end{enumerate}
    \textbf{end}\\
    \textbf{return} $z^*=z$;
\end{algorithm}

\section{$O(n^3)$ Feasible IPM Algorithm}
Reducing the certified $O(\sqrt{n})$ iteration complexity to $O(\log(n))$ remains an open challenge in the development of IPMs. If achieved, it would reduce the total time complexity of IPMs to $O(n^3\log(n))\approx O(n^3)$. 

Rather than aiming to reduce the number of iterations, this article shifts attention to lowering the cost per iteration by exploiting the similarity between adjacent linear systems. This similarity comes from the conservative $O(\sqrt{n})$ iteration complexity and is exploited by the proposed \emph{approximated Newton step}, which consists of multiple rank-1 updates.

\subsection{Approximated Newton step}
Instead of solving the linear Newton system \eqref{eqn_KKT_delta}, the \emph{approximated Newton step} introduces the approximate 
\[
(\Tilde{x},\Tilde{s})\triangleq(\tilde{\gamma},\tilde{\theta},\tilde{\phi},\tilde{\psi})
\]
for $(x,s)=(\gamma,\theta,\phi,\psi)$ into the Newton system \eqref{eqn_KKT_delta}. That is,
\begin{subequations}\label{eqn_approximate_Newton_step}
\begin{align}
2\lambda\tilde{H}\Delta z +\Omega\Delta x &= 0,\label{eqn_approximate_a}\\
    \Omega^\top\Delta z + \Delta s &=0,\label{eqn_approximate_b}\\
\tilde{s}\pdot\Delta x+ \tilde{x}\pdot\Delta s &= \tau e- x\pdot s.\label{eqn_approximate_c}
\end{align}
\end{subequations}
Similar to \eqref{eqn_Delta_x_Delta_s_positive} (as $H\succeq0$), we have that 
\[
 \Delta x^\top \Delta s = \Delta x^\top (-\Omega^\top \Delta z)=\Delta z^\top(2\lambda\tilde{H})\Delta z\geq0.
\]
 By letting
\begin{subequations}\label{eqn_approximated_delta_gamma_etc_update}
    \begin{align}
        &\Delta \gamma=\frac{\tilde{\gamma}}{\tilde{\phi}}\Delta z+\tau\frac{1}{\tilde{\phi}}- \frac{\gamma\pdot\phi}{\tilde{\phi}},\\
        &\Delta \theta=-\frac{\tilde{\theta}}{\tilde{\psi}}\Delta z+\tau\frac{1}{\tilde{\psi}}-\frac{\theta\pdot\psi}{\tilde{\psi}},\\
        &\Delta\phi = - \Delta z,\\
        &\Delta\psi = \Delta z,
    \end{align}
\end{subequations}
\eqref{eqn_approximate_Newton_step} can be reduced into a more compact system of linear equations,
\begin{equation}\label{eqn_approximated_Newton_compact_system}
\Big(2\lambda\tilde{H}+\diag\!\Big(\frac{\tilde{\gamma}}{\tilde{\phi}}\Big) + \diag\!\Big(\frac{\tilde{\theta}}{\tilde{\psi}}\Big) \!\Big) \Delta z=\tau\frac{1}{\tilde{\psi}}-\tau\frac{1}{\tilde{\phi}}+ \frac{\gamma\pdot\phi}{\tilde{\phi}}- \frac{\theta\pdot\psi}{\tilde{\psi}}.
\end{equation}
Unlike the compact system \eqref{eqn_compact_linsys} in the exact Newton step \eqref{eqn_KKT_delta} typically solved using the $O(n^3)$ Cholesky factorization, the compact system \eqref{eqn_approximated_Newton_compact_system} in the \textit{approximated Newton step} \eqref{eqn_approximate_Newton_step} leverages multiple rank-1 update techniques based on the introduced update rule for $(\Tilde{x}, \Tilde{s})$.

The update rule for the approximate $(\Tilde{x},\Tilde{s})$ is that: \textit{i)} before starting iterations,  $(\tilde{x},\tilde{s})\leftarrow(x,s)$; \textit{ii)} at each iteration once $(x,s)$ updates, then $(\tilde{x},\tilde{s})$ is updated as
\begin{equation}\label{eqn_update_rule_approximate_x_s}
    \begin{aligned}
        &\tilde{x}_i \leftarrow\left \{
        \begin{array}{l}
x_i,~\textrm{for }
\frac{\tilde{x}_i}{x_i} \notin\!\left[ \frac{1}{1+\delta}, 1+\delta\right], \\
\tilde{x}_i, ~\textrm{otherwise},
    \end{array}\right.~ i=1,2,\cdots{},2n\\
&\tilde{s}_i\leftarrow \left\{\begin{array}{l}
s_i,~\textrm{for } \frac{\Tilde{s}_i}{s_i}\notin \! \left[ \frac{1}{1+\delta}, 1+\delta\right], \\
\tilde{s}_i, ~\textrm{otherwise},
\end{array}\right.~ i=1,2,\cdots{},2n
\end{aligned}
\end{equation}
where $\delta>0$ is the approximation parameter (its value will be discussed later). From the update rule \eqref{eqn_update_rule_approximate_x_s} for the approximate $(\tilde{x},\tilde{s})$, it is easy to derive the conditions:
\[
    \begin{aligned}
    &\frac{\tilde{x}_i}{x_i} \in \!\left[\frac{1}{1+\delta},1+ \delta\right], \quad i=1,2,\cdots{},2n,\\
&\frac{\tilde{s}_i}{s_i} \in \!\left[\frac{1}{1+\delta},1+\delta \right], \quad i=1,2,\cdots{},2n,
    \end{aligned}
\]
which is crucial for proving the time complexity in subsequent subsections.

Moreover, expanding \eqref{eqn_update_rule_approximate_x_s} gives that
\begin{equation}\label{eqn_tilde_gamma_theta_phi_psi_update}
    \begin{aligned}
&\tilde{\gamma}_i\leftarrow\left\{\begin{array}{l}
\gamma_i,~\textrm{for } \frac{\tilde{\gamma}_i}{\gamma_i},\notin\! 
\left[\frac{1}{1+\delta}, 1+\delta\right],\\
\tilde{\gamma}_i,~\textrm{otherwise},
\end{array}\right.~ i=1,2,\cdots{},n\\
&\tilde{\theta}_i\leftarrow 
\left\{\begin{array}{l}
 \theta_i,~\textrm{for } \frac{\tilde{\theta}_i}{\theta_i}\notin\! \left[\frac{1}{1+\delta},1+\delta\right], \\
\tilde{\theta}_i,~\textrm{otherwise},
 \end{array}\right.~ i=1,2,\cdots{},n\\
&\tilde{\phi}_i\leftarrow\left\{\begin{array}{l}
\phi_i,~\textrm{for } \frac{\tilde{\phi}_i}{\phi_i}\notin \left[\frac{1}{1+\delta},1+\delta\right], \\
\tilde{\phi}_i,~\textrm{otherwise},
\end{array}\right.~ i=1,2,\cdots{},n\\
&\tilde{\psi}_i\leftarrow 
\left\{\begin{array}{l}
\psi_i,~\textrm{for } \frac{\tilde{\psi}_i}{\psi_i}\notin \left[\frac{1}{1+\delta},1+\delta\right], \\
\tilde{\psi}_i,~\textrm{otherwise}
\end{array}\right.~ i=1,2,\cdots{},n\\
    \end{aligned}
\end{equation}
Before showing how to exploit \eqref{eqn_tilde_gamma_theta_phi_psi_update} to solve the compact system  \eqref{eqn_approximated_Newton_compact_system} cheaply, we first introduce the vector $d$ and the matrix $B$ as
\begin{subequations}
    \begin{align}
    d&\triangleq\frac{\tilde{\gamma}}{\tilde{\phi}}+\frac{\tilde{\theta}}{\tilde{\psi}},\\
    B&\triangleq2\lambda\tilde{H}+\diag(d).   
    \end{align}
\end{subequations}
Below are some sets used in the derivation of the update rule for $B$ and $d$.
\begin{definition}
  Let $I_{\gamma}, I_{\theta}, I_{\phi}, I_{\psi}$ denote the sets of indices $\{i\left|i=1,2,\cdots{},n\right.\}$ as
\begin{equation}\label{eqn_sets_I_gamma_theta_phi_psi}
  \begin{aligned}
   I_{\gamma} &\triangleq\left\{i \,\left|\, \frac{\tilde{\gamma}_i}{\gamma_i}\notin\! \left[\frac{1}{1+\delta},1+\delta\right],\ i=1,2,\cdots{},n \right.\right\}  \\  
    I_\theta &\triangleq\left\{i \,\left|\, \frac{\tilde{\theta}_i}{\theta_i}\notin\! \left[\frac{1}{1+\delta},1+\delta\right],\ i=1,2,\cdots{},n \right.\right\}\\
    I_{\phi} &\triangleq\left\{i \,\left|\, \frac{\tilde{\phi}_i}{\phi_i}\notin\! \left[\frac{1}{1+\delta},1+\delta\right],\ i=1,2,\cdots{},n \right.\right\}  \\  
    I_\psi &\triangleq\left\{i \,\left|\, \frac{\tilde{\psi}_i}{\psi_i}\notin\! \left[\frac{1}{1+\delta}, 1+\delta\right],\ i=1,2,\cdots{},n \right.\right\}  
    \end{aligned}
  \end{equation}
for the current iterate $(z,\gamma,\theta,\phi,\psi)$ and $(\tilde{\gamma},\tilde{\theta},\tilde{\phi},\tilde{\psi})$.
\end{definition}
Then, an implementable procedure for  \eqref{eqn_tilde_gamma_theta_phi_psi_update} and the update rule for sets $I_\gamma,I_\theta,I_\phi,I_\psi$ is formulated as
\begin{subequations}\label{procedure_gamma_theta_phi_psi}
    \begin{align}
&I_\gamma\leftarrow\emptyset,~I_\theta\leftarrow\emptyset,~I_\phi\leftarrow\emptyset,~I_\psi\leftarrow\emptyset;\\
    &\textbf{for}~i=1,2,\cdots{},n\notag\\
&~\mathrm{if}~\frac{\tilde{\gamma}_i}{\gamma_i}\notin\! \left[\frac{1}{1+\delta},1+\delta\right]: \, \tilde{\gamma}_i\leftarrow \gamma_i, I_\gamma\leftarrow I_\gamma\cup\left\{i\right\};\\
    &~\mathrm{if}~\frac{\tilde{\theta}_i}{\theta_i}\notin\! \left[\frac{1}{1+\delta},1+\delta\right]:\, \tilde{\theta}_i\leftarrow \theta_i, I_\theta\leftarrow I_\theta\cup\left\{i\right\};\\
    &~\mathrm{if}~\frac{\tilde{\phi}_i}{\phi_i}\notin\! \left[\frac{1}{1+\delta},1+\delta\right]:\, \tilde{\gamma}_i\leftarrow \gamma_i, I_\phi\leftarrow I_\phi\cup\left\{i\right\};\\
    &~\mathrm{if}~\frac{\Tilde{\psi}_i}{\psi_i}\notin\! \left[\frac{1}{1+\delta},1+\delta\right]:\, \tilde{\psi}_i\leftarrow \psi_i, I_\psi\leftarrow I_\psi\cup\left\{i\right\};\\
        &\textbf{end}\notag
    \end{align}
\end{subequations}
After executing the above procedure \eqref{procedure_gamma_theta_phi_psi}, the update rules for $B$ and $d$ are 
\begin{subequations}\label{eqn_B_update}
    \begin{align}
        &B \leftarrow B+ \sum_{i\in I_\gamma\cup I_\theta\cup I_\phi\cup I_\psi} \! \left(\frac{\tilde{\gamma}_i}{\tilde{\phi}_i}+\frac{\tilde{\theta}_i}{\tilde{\psi}_i}-d_i\right)\!e_ie_i^\top,\\
        &d \leftarrow \frac{\tilde{\gamma}}{\tilde{\phi}}+\frac{\tilde{\theta}}{\tilde{\psi}}.
    \end{align}
\end{subequations}
Solving the compact system \eqref{eqn_approximated_Newton_compact_system} requires maintaining the inverse of $B$, instead of $B$ itself. Thus, we introduce the matrix
\begin{equation}
    M\triangleq B^{-1}.
\end{equation}
Then, based on the well-known Sherman-Morrison formula and  \eqref{eqn_B_update}, the update rules of $M$ and $d$ are
\begin{subequations}\label{eqn_M_update}
    \begin{align}
        &\textbf{for}~i\in I_\gamma\cup I_\theta\cup I_\phi\cup I_\psi\notag\\
        &\quad\Delta \leftarrow \frac{\tilde{\gamma}_i}{\tilde{\phi}_i}+\frac{\tilde{\theta}_i}{\tilde{\psi}_i}-d_i\\
        &\quad M\leftarrow M -\frac{\Delta}{1+\Delta \cdot M_{i,i}}M_{\cdot,i}M_{\cdot,i}^\top\label{eqn_M_update_b}\\
        &\quad d_i \leftarrow \frac{\tilde{\gamma_i}}{\tilde{\phi_i}}+\frac{\tilde{\theta_i}}{\tilde{\psi_i}}\\
        &\textbf{end}\notag
    \end{align}
\end{subequations}
\begin{remark}
As the update of $M$, \eqref{eqn_M_update} involves multiple rank-1 updates and each rank-1 update requires $O(n^2)$ operations. Thus,  \eqref{eqn_M_update} totally requires $O\left(|I_\gamma\cup I_\theta\cup I_\phi\cup I_\psi|\times n^2\right)$ operations.
\end{remark}

After updating $M$, calculating $\Delta z$ from \eqref{eqn_approximated_Newton_compact_system} is a matrix-vector multiplication,
\begin{equation}\label{eqn_delta_z_update}
    \Delta z = M\!\left(\tau\frac{1}{\tilde{\psi}}-\tau\frac{1}{\tilde{\phi}}+ \frac{\gamma\pdot\phi}{\tilde{\phi}}- \frac{\theta\pdot\psi}{\tilde{\psi}}\right),
\end{equation}
and then $(\Delta \gamma, \Delta\theta,\Delta\phi,\Delta\psi)$ can be calculated from  \eqref{eqn_approximated_delta_gamma_etc_update}.

The next two subsections prove the iteration complexity and the total number of rank-1 updates during the total iterations, respectively.

\subsection{$O(\sqrt{n})$ iteration complexity}
One main contribution of this article is the following theorem, which demonstrates that using the \emph{approximated Newton step} in the feasible IPM algorithm framework yields the same $O(\sqrt{n})$ iteration complexity as using the exact Newton step.
\begin{theorem}\label{theorem_n_3_iteration}
Consider any values of $\alpha\in(0,1),\delta>0$ that satisfy\footnote{For example, choosing $\alpha=0.3$, $\delta=0.15$ results in $\mu=0.6518\in(0,1), ~\sigma=0.1997<\alpha$ (which satisfies $\beta>0$).} 
\begin{equation}\label{eqn_alpha_delta_beta}
        \begin{aligned}
        \mu&=\frac{(1+\delta)^3\alpha}{1-\alpha}\in(0,1),\\
        \beta&=\frac{\alpha-\sigma}{1+\frac{\alpha}{\sqrt{2n}}}>0,
    \end{aligned}
    \end{equation}
    where 
    \[
\sigma=\sqrt{2}\delta(1+\delta)^2\alpha\sqrt{\frac{1+\alpha}{1-\alpha}}+\frac{(1+\delta)^2\alpha^2}{2(1-\alpha)}.
    \]
    
   Suppose that 
    \begin{subequations}
        \begin{align}
            &(x,s)>0,\quad \xi\triangleq\frac{\|x\pdot s-\tau e\|}{\tau}\leq\alpha , \label{eqn_assumption_cen}\\
            &\frac{\tilde{x}_i}{x_i} \in\! \left[\frac{1}{1+\delta},1+\delta\right], \quad i=1,\cdots{},2n,\label{eqn_tilde_x_definition}\\
            &\frac{\tilde{s}_i}{s_i} \in\! \left[\frac{1}{1+\delta},1+\delta\right], \quad i=1,\cdots{},2n.\label{eqn_tilde_s_definition}
        \end{align}
    \end{subequations}
    Using the update rule \eqref{eqn_assumption_tau} for $\tau$, the next iterate $(z^+,x^+,s^+)$, given by
    \begin{equation}\label{eqn_next_point}
        (z^+, x^+,s^+) = (z, x,s) + (\Delta z, \Delta x, \Delta s)
    \end{equation}
   where $(\Delta z, \Delta x, \Delta s)$ is the solution of \eqref{eqn_approximate_Newton_step}, satisfies
\begin{equation} \label{eqn_next_iterate_center}
(x^+,s^+)>0,~\quad \xi^+\triangleq\frac{\|x^+ \pdot s^+-\tau^+ e\|}{\tau^+}\leq\alpha.
\end{equation}
    Furthermore, after \textbf{exact}
    \begin{equation}\label{eqn_number_iterations}
    \mathcal{N}_\mathrm{iter}=\!\left\lceil\frac{\log\!\left(\frac{2n+\alpha\sqrt{2n}}{\epsilon}\right)}{-\log\!\left(1-\frac{\beta}{\sqrt{2n}}\right)}\right\rceil
    \end{equation}
    iterations, the inequality $x^\top s\leq\epsilon$ holds.
\end{theorem}
\begin{proof}
    See Appendix \ref{sec_theorem_sqrt_n}.
\end{proof}

\subsection{$O(n)$ number of rank-1 updates} 
Another main contribution of this article is the analysis of the total number of rank-1 updates during the whole $\mathcal{N}_\mathrm{iter}$ iterations (given by \eqref{eqn_number_iterations}), which is proven to be bounded by $O(n)$, as presented in the below theorem.

Let $I_{\gamma}^k,~I_{\theta}^k,~I_{\phi}^k,~I_{\psi}^k$ denote the $I_{\gamma},~I_{\theta},~I_{\phi},~I_{\psi}$ (from the definition \eqref{eqn_sets_I_gamma_theta_phi_psi}) at the $k$th iterate, respectively.

\begin{theorem}\label{theorem_num_rank}
The total number of rank-1 updates required by Algorithm \ref{alg_n_3} is bounded, 
\begin{equation}\label{eqn_number_ranks}
    \mathcal{N}_{\mathrm{rank}-1}=\sum_{k=1}^{\mathcal{N}}|I_{\gamma}^k\cup I_{\theta}^k\cup I_{\phi}^k\cup I_{\psi}^k|\leq\left\lceil\frac{4\eta(\mathcal{N}_\mathrm{iter}-1)\sqrt{n}}{\left(1-\eta\right)\log(1+\delta)}\right\rceil.
\end{equation}
\end{theorem}
\begin{proof}
    See Appendix \ref{sec_theorem_num_rank}.
\end{proof}

\subsection{Algorithm Implementation}
Based on the same framework (feasible IPMs) as Algorithm \ref{alg_n_3_dot_5} (using the same initialization), Theorems \ref{theorem_n_3_iteration} and  \ref{theorem_num_rank}, we outline an algorithm that adopts the \textit{approximated Newton step}, as shown in Algorithm \ref{alg_n_3}.

\begin{theorem}
 The total time complexity of Algorithm \ref{alg_n_3} is $O(n^3)$.   
\end{theorem}
\begin{proof}
    First, Algorithm \ref{alg_n_3} exactly requires $\mathcal{N}_\mathrm{iter}$ iterations, given in \eqref{eqn_number_iterations}, which
    is proved to be $O(\sqrt{n})$ (as $-\log(1-\beta/\sqrt{2n})\geq\beta/\sqrt{2n}$ when $0<\beta/\sqrt{2n}<1$). Then, by Theorem \ref{theorem_num_rank}, the upper bound of total number of 
    rank-1 updates 
    is given by \eqref{eqn_number_ranks}, which is $O(\mathcal{N}_\mathrm{iter}\times \sqrt{n})=O(n)$. The time complexity of one rank-1 update is $O(n^2)$. Therefore, the total time complexity of Algorithm  \ref{alg_n_3} is $O(n^3)$, which completes the proof.
\end{proof}

\begin{algorithm}
    \caption{An $O(n^3)$ Feasible IPM Algorithm for }\label{alg_n_3}
    \textbf{Input}: Given the Box-QP data $(H,h)$, the dimension $n$. 
    \vspace*{.1cm}\hrule\vspace*{.1cm}
    \textbf{if }$\|h\|_\infty=0$, \textbf{return} $z^*=0$;\\
    \textbf{otherwise},\\
    Choose $\alpha=0.3$, $\delta=0.15$, and an optimality level $\epsilon$ (such as $\epsilon=10^{-6}$). Set $\lambda=\frac{\alpha}{\sqrt{2n}}$. Cache $2\lambda\tilde{H}=\frac{2\lambda}{\|h\|_\infty}H$ and $\tilde{h}=\frac{1}{\|h\|_\infty}h$. Initialize $(z,\gamma,\theta,\phi,\psi)$ from \eqref{eqn_initialization_stragegy} and $\tau=1$. Compute $\sigma=\sqrt{2}\delta(1+\delta)^2\alpha\sqrt{\frac{1+\alpha}{1-\alpha}}+\frac{(1+\delta)^2\alpha^2}{2(1-\alpha)}$ and $\beta=\frac{\alpha-\sigma}{1+\frac{\alpha}{\sqrt{2n}}}$, and the required number of iterations  $\mathcal{N}_\mathrm{iter}=\!\left\lceil\frac{\log\left(\frac{2n+\alpha\sqrt{2n}}{\epsilon}\right)}{-\log\left(1-\frac{\beta}{\sqrt{2n}}\right)}\right\rceil$;
    \\ \\
$\tilde{\gamma}\leftarrow\gamma,~\tilde{\theta}\leftarrow\theta,~\tilde{\phi}\leftarrow\phi,~\tilde{\psi}\leftarrow\psi$;\\  \\
    $d=\frac{\tilde{\gamma}}{\phi}+\frac{\tilde{\theta}}{\psi}$;\\ \\
$M=\left(2\lambda\tilde{H}+\diag(d)\right)^{\!-1}$;\\ \\
    \textbf{for} $k=1,2,\cdots{},\mathcal{N}_\mathrm{iter}$, then 
    \begin{enumerate}[label*=\arabic*., ref=\theenumi{}]       
        \item Update $(\tilde{\gamma},~\tilde{\theta},~\tilde{\phi},~\tilde{\psi})$ and assign the sets $(I_{\gamma},I_{\theta},I_{\phi},I_{\psi})$ as per \eqref{procedure_gamma_theta_phi_psi}\\
        \item Update $M$ and $d$ as per \eqref{eqn_M_update}\\
        \item Compute $\Delta z$ as per \eqref{eqn_delta_z_update} and $(\Delta \gamma, \Delta\theta,\Delta\phi,\Delta\psi)$ as per \eqref{eqn_approximated_delta_gamma_etc_update}\\
        \item $(z,\gamma,\theta,\phi,\psi)\leftarrow(z,\gamma,\theta,\phi,\psi) + (\Delta z,\Delta\gamma,\Delta\theta,\Delta\phi,\Delta\psi)$;\\
         \item $\tau\leftarrow\left(1-\frac{\beta}{\sqrt{2n}}\right)\tau$;
    \end{enumerate}
    \textbf{end}\\
    \textbf{return} $z^*=z$;
\end{algorithm}


\section{Numerical Validations}
Algorithms \ref{alg_n_3_dot_5} and \ref{alg_n_3} are implemented in C code and executed in Julia via the Julia-C interface ccall. Code is available at \url{https://github.com/liangwu2019/N3QP}. The reported numerical results were obtained on a Dell Precision 7760 (equipped with 2.5 GHz 16-core Intel Core i7-11850H and 128GB RAM) running Julia v1.10 on Ubuntu 22.04.

\subsection{Validation of number of iterations and rank-1 updates}
To validate Theorems \ref{theorem_n_3_iteration} and \ref{theorem_num_rank}, namely verifying that the number of iterations and rank-1 updates in Algorithm \ref{alg_n_3} is below $\mathcal{N}_\mathrm{iter}$ and $\mathcal{N}_{\mathrm{rank}-1}$, which are only dependent of the problem dimension $n$ and the optimality level $\epsilon$, we first generate random Box-QP examples with different sizes ($n=100, 200$). Then, we examine how many number of iterations and rank-1 updates are required by Algorithm \ref{alg_n_3} to solve Box-QPs ($n=100, 200$) to the optimality level $\epsilon=10^{-6}, 10^{-8}$.

Algorithm \ref{alg_n_3} belongs to feasible IPMs and thus the duality gap can represent the accuracy of iterates. 

Figs.\ \ref{fig_num_iters} and  \ref{fig_num_ranks} describe the decreasing relationship between the duality gap and the number of iterations and the number of rank-1 updates, respectively. Fig.\ \ref{fig_num_iters} demonstrates that Algorithm \ref{alg_n_3} arrives at the optimality level $\epsilon=10^{-6},10^{-8}$ in the \textbf{same} iterations with the theoretical bound for both $n=100$ and $200$. Fig.\ \ref{fig_num_ranks} demonstrates that Algorithm \ref{alg_n_3} requires \textbf{fewer} rank-1 updates than the theoretical bound for both $n=100$ and $200$. This numerically validates Theorems \ref{theorem_n_3_iteration} and  \ref{theorem_num_rank} for this case study.

\begin{figure}[!htbp]
    \centering
    \includegraphics[width=1.0\linewidth]{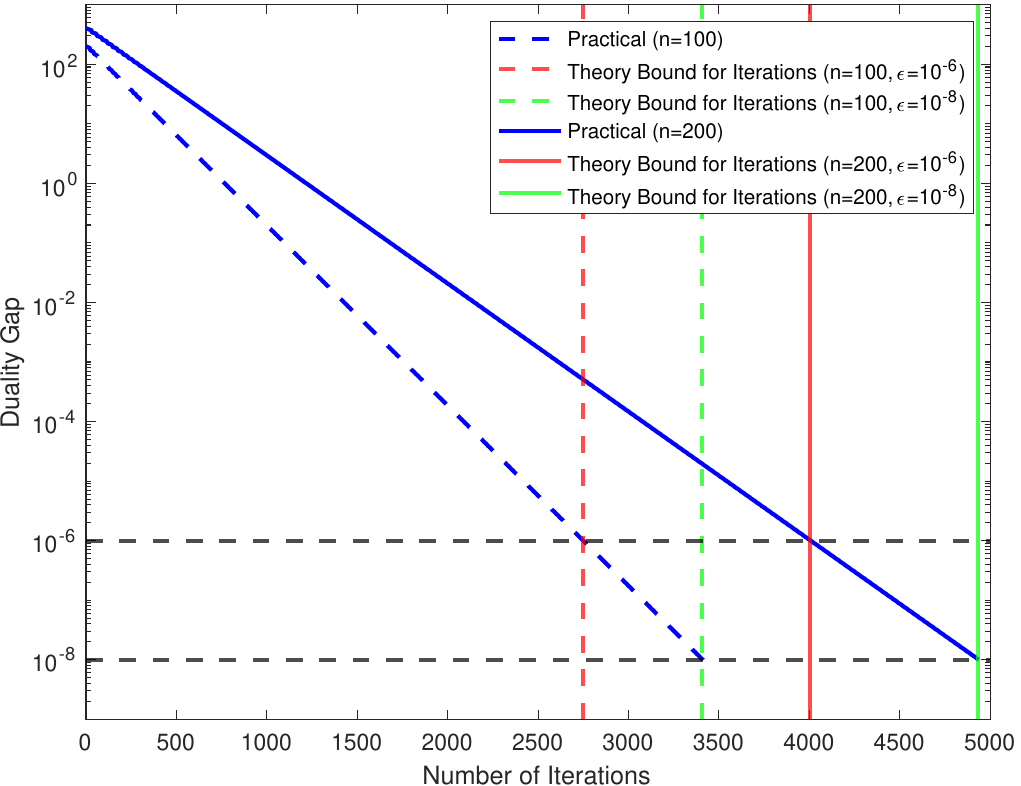}
    
    \vspace{-0.2cm} 
    \caption{Duality gap vs.\ the number of iterations and $\mathcal{N}_\mathrm{iter}$ (the theory bound for the number of iterations) by applying  Algorithm \ref{alg_n_3} to Box-QP for $n=100,200$ and $\epsilon=10^{-6},10^{-8}$.}
    \label{fig_num_iters}
\end{figure}

\begin{figure}[!htbp]
    \centering
    \includegraphics[width=1.0\linewidth]{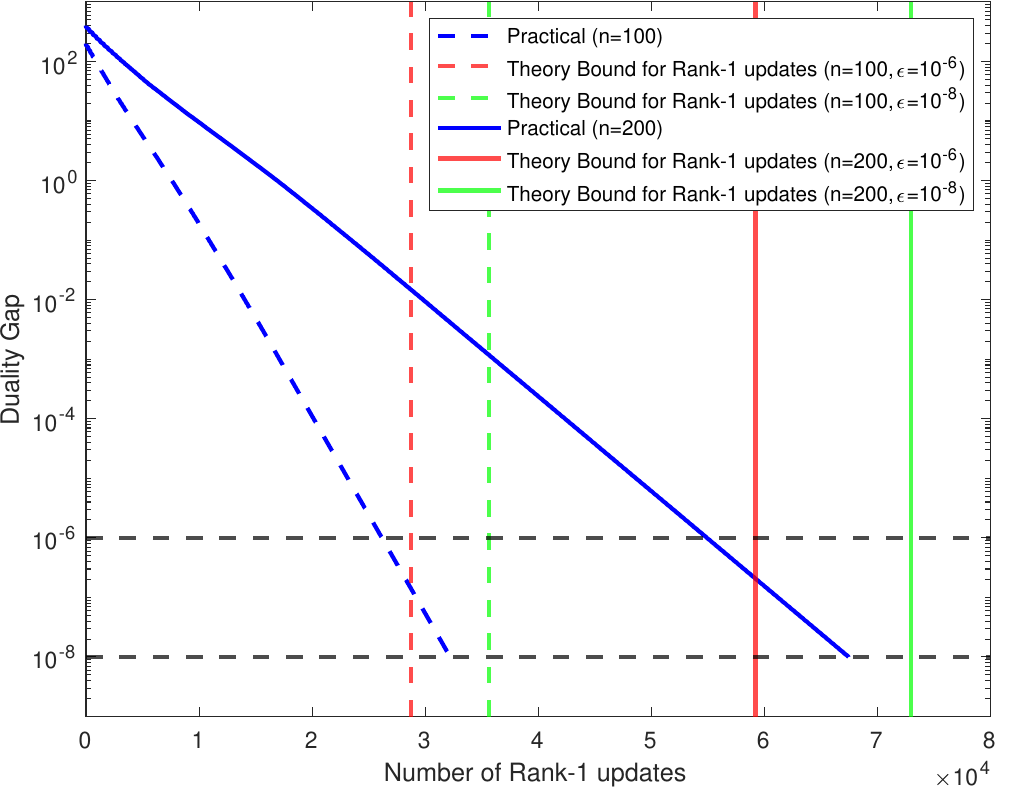}
    
    \vspace{-0.3cm} 
    \caption{Duality gap vs.\ the number of rank-1 updates and $\mathcal{N}_{\mathrm{rank}-1}$ (the theory bound for the number of rank-1 updates) by applying  Algorithm \ref{alg_n_3} to Box-QP for $n=100,200$ and $\epsilon=10^{-6},10^{-8}$.}
    \label{fig_num_ranks}
\end{figure}

\subsection{Numerical comparison of Algorithms \ref{alg_n_3_dot_5} and \ref{alg_n_3}}
It is well known that solving linear systems usually has the $O(n^3)$ time complexity, while solving QPs via IPMs is more expensive with  $O(n^{3.5})$  time complexity as in Algorithm \ref{alg_n_3_dot_5}. One main contribution of this article is the development of an implementable Algorithm \ref{alg_n_3}, which achieves the $O(n^3)$ time complexity by using the \textit{approximated Newton step}. Moreover, the approximated Newton step eliminates the need for a linear system solver (e.g., Cholesky factorization), rendering Algorithm \ref{alg_n_3} \textit{matrix-free} (or \textit{inverse-free}) during iterations. While a one-time matrix inversion is required before starting iterations, it can be precomputed offline. Algorithm \ref{alg_n_3} then relies solely on rank-1 updates of a symmetric positive definite matrix as shown in  \eqref{eqn_M_update_b}: $M\leftarrow M -\frac{\Delta}{1+\Delta \cdot M_{i,i}}M_{\cdot,i}M_{\cdot,i}^\top$. 

While Algorithm \ref{alg_n_3} has a lower time complexity than Algorithm \ref{alg_n_3_dot_5} in theory, their practical computing performance depends on the implementation of their most expensive steps: the rank-1 updates and Cholesky factorization of a symmetric positive definite matrix in Algorithms \ref{alg_n_3} and \ref{alg_n_3_dot_5}, respectively. To do a fair numerical comparison, the implementations of Algorithm \ref{alg_n_3} and Algorithm \ref{alg_n_3_dot_5} are based on the same OpenBLAS v0.3.30 library on the same computing platform. Specifically, in our implementation,
\begin{itemize}
    \item Algorithm \ref{alg_n_3_dot_5} calls the \textbf{LAPACKE\_dposv()} function for the Cholesky factorization step;
    \item Algorithm \ref{alg_n_3} calls the \textbf{cblas\_dsyr()} function for the rank-1 updates step.
\end{itemize}
In Fig.\  \ref{fig_run_time}, these variants are denoted as \textit{Algorithm \ref{alg_n_3_dot_5} (dposv)} and \textit{Algorithm \ref{alg_n_3} (dsyr)}, respectively. Fig.\  \ref{fig_run_time} plots their CPU time (s) spent on solving randomly generated Box-QP problems with different problem dimensions ranging from $n=50$ to $1000$. When $n\geq100$, Algorithm \ref{alg_n_3} has a smaller computation time than Algorithm \ref{alg_n_3_dot_5}, which is consistent with the theoretical complexity: $O(n^3)$ is faster than $O(n^{3.5})$. In our practical implementation, Algorithm \ref{alg_n_3} achieves around 3-4 fold speedup over Algorithm \ref{alg_n_3_dot_5}.

\begin{figure}[!htbp]
    \centering
    \includegraphics[width=1.0\linewidth]{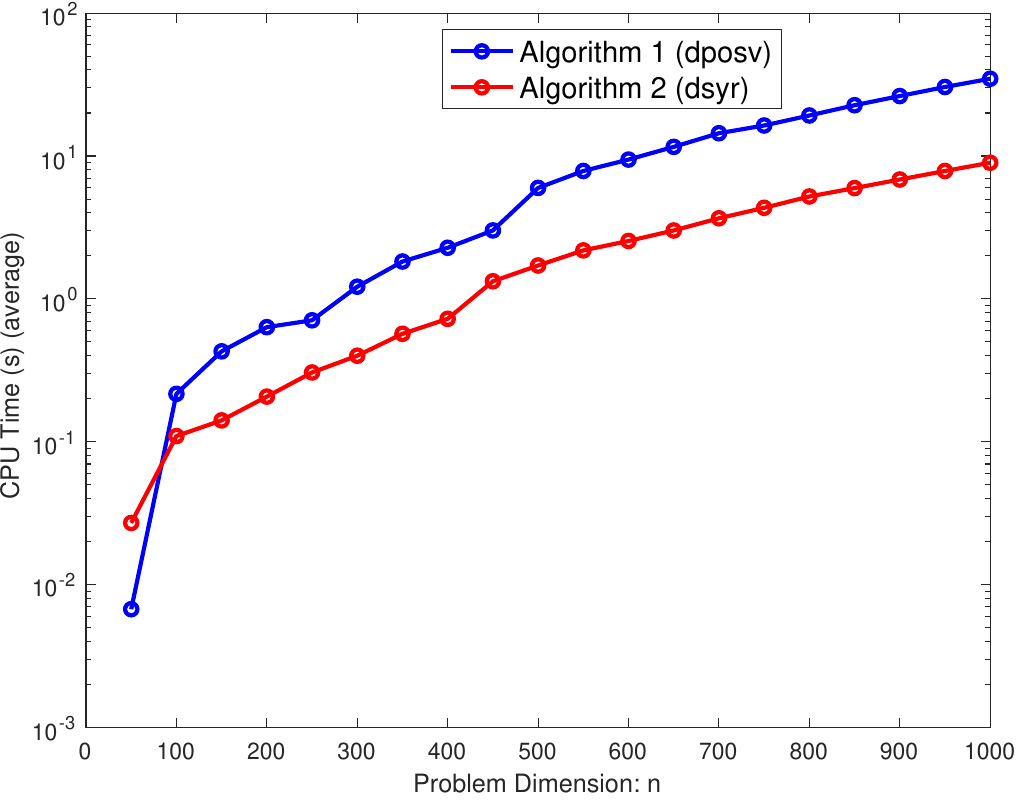}

\vspace{-0.2cm}
    
    \caption{CPU time (s) spent on random Box-QP problems with different dimensions.}
    \label{fig_run_time}
\end{figure}

\section{Applications}
\subsection{Online $\ell_1$-penalty soft-constrained MPC}
Among the industrial practice of online MPC, two issues are often encountered: \textit{i)} possible online infeasibility due to unknown disturbances or modeling errors, and \textit{ii)} a missing execution time certificate for solving the MPC problem. To solve the two issues simultaneously, this subsection adopts the $\ell_1$ penalty presented in Subsection \ref{sec_ell_1_penalty}, which transforms the resulting strictly convex QP (from condensing MPC formulation) to a Box-QP, and our proposed Algorithm \ref{alg_n_3} offers certified computation time complexity for Box-QP problems.

We consider linear MPC problems for tracking,
\begin{subequations}\label{problem_MPC_tracking}
    \begin{align}
        \min~&\frac{1}{2}\sum_{k=0}^{N_p-1}\left\|W_y\! \left(Cx_{k+1}-r(t)\right)\right\|_{2}^{2} + \left\|W_{\Delta u}  \Delta u_{k}\right\|_{2}^{2}\label{problem_MPC_tracking_obj}\\
        \mathrm{s.t.}~& x_{k+1}=A x_{k}+B u_{k}, ~k=0, \ldots, N_p-1,\label{problem_MPC_tracking_A_B}\\
        & u_{k} = u_{k-1} + \Delta u_{k},~k=0, \ldots, N_p-1,\\
        & E_x x_{k+1} + E_u u_k + E_{\Delta u} \Delta u_{k}\leq f_k,~k=0, \ldots, N_p-1,\label{problem_MPC_tracking_constraint}\\
        &x_0 = x(t), u_{-1} = u(t-1),
    \end{align}
\end{subequations}
where $x_k\in\mathbb{R}^{n_x},u_k\in\mathbb{R}^{n_u},\Delta u_k$ denote the states, the control inputs, and the control input increments along the prediction horizon $N_p$. The control goal is to minimize the tracking errors between the output $Cx_{k+1}$ and the reference signal $r(t)$ ($W_y\succ0, W_{\Delta u}\succ0$ denotes the weights for output tracking error and control input increments, respectively). $E_x x_{k+1} + E_u u_k + E_{\Delta u} \Delta u_{k}\leq f_k$ denotes the general constraints formulation such as the box constraints $u_{\min } \leq u_{k} \leq u_{\max}, x_{\min} \leq x_{k+1} \leq x_{\max},\Delta u_{\min } \leq \Delta u_{k} \leq \Delta u_{\max}$ or the terminal constraints $x_{N_p}\in\mathcal{X}_{N_p}$.

Consider the \textit{AFTI-16 MPC problem with initial infeasibility} as an example. The model \eqref{problem_MPC_tracking_A_B} in MPC is obtained by discretizing (with sampling time $0.05$ s) the open-loop unstable AFTI-F16 aircraft dynamics (from \cite{bemporad1997nonlinear}),
\[
\left\{\begin{aligned}
\dot{x} =&{\left[\begin{array}{cccc}
-0.0151 & -60.5651 & 0 & -32.174 \\
-0.0001 & -1.3411 & 0.9929 & 0 \\
0.00018 & 43.2541 & -0.86939 & 0 \\
0 & 0 & 1 & 0
\end{array}\right]} x\\&+{\left[\begin{array}{cc}
-2.516 & -13.136 \\
-0.1689 & -0.2514 \\
-17.251 & -1.5766 \\
0 & 0
\end{array}\right] }u \\
y =&{\left[\begin{array}{llll}
0 & 1 & 0 & 0 \\
0 & 0 & 0 & 1
\end{array}\right]\!x}
\end{aligned}\right.
\]
The constraints \eqref{problem_MPC_tracking_constraint} come from the input,  $|u_i| \leq 25^{\circ}$, $i = 1, 2$, and the output, $-0.5\leq y_1 \leq 0.5 $ and $-100 \leq y_2 \leq 100$. The control objective \eqref{problem_MPC_tracking_obj} is to regulate the attach angle $y_1$ to zero and make the pitch angle $y_2$ track a reference signal $r_2$ with the cost weight $W_y = \diag([10,10]), W_{\Delta u}= \diag([0.1, 0.1])$ and the prediction horizon $N_p=5$. 

By eliminating the states $\{x_k\}$, the MPC problem \eqref{problem_MPC_tracking} is condensed into a strictly convex QP \eqref{eqn_standard_QP}. When the current state already violates the output constraint (for example, $x(0)=[0,5,0,0]^\top$ ($y_1(0)=x_2(0)=5$) violates $-0.5\leq y_1 \leq 0.5$), the QP \eqref{eqn_standard_QP} has no feasible solution. Then, the QP \eqref{eqn_standard_QP} is softened to a Box-QP \eqref{eqn_box_QP} by adopting the $\ell_1$-penalty scheme ($\rho=10^3$ or $10^4$ for the output and input constraints), see Subsection \ref{sec_ell_1_penalty}. 

The resulting Box-QP has the dimension $n=40$ (the number of constraints is $4$$\times$$2$$\times$$N_p$). By choosing the optimality level $\epsilon=10^{-6}$, we can know in advance that, in Algorithm \ref{alg_n_3}, the upper bound of the number of iterations is $\mathcal{N}_\mathrm{iter}=1672$ and the number of rank-1 updates is $\mathcal{N}_{\mathrm{rank}-1}=11060$, which can in advance certify that the execution time in a choosing computing platform will be smaller than the chosen sampling time, $0.05$ s. For example, our computing platform takes $0.011$ s to do the calculation..

\begin{figure}[!htbp]
    \centering
    \includegraphics[width=1\linewidth]{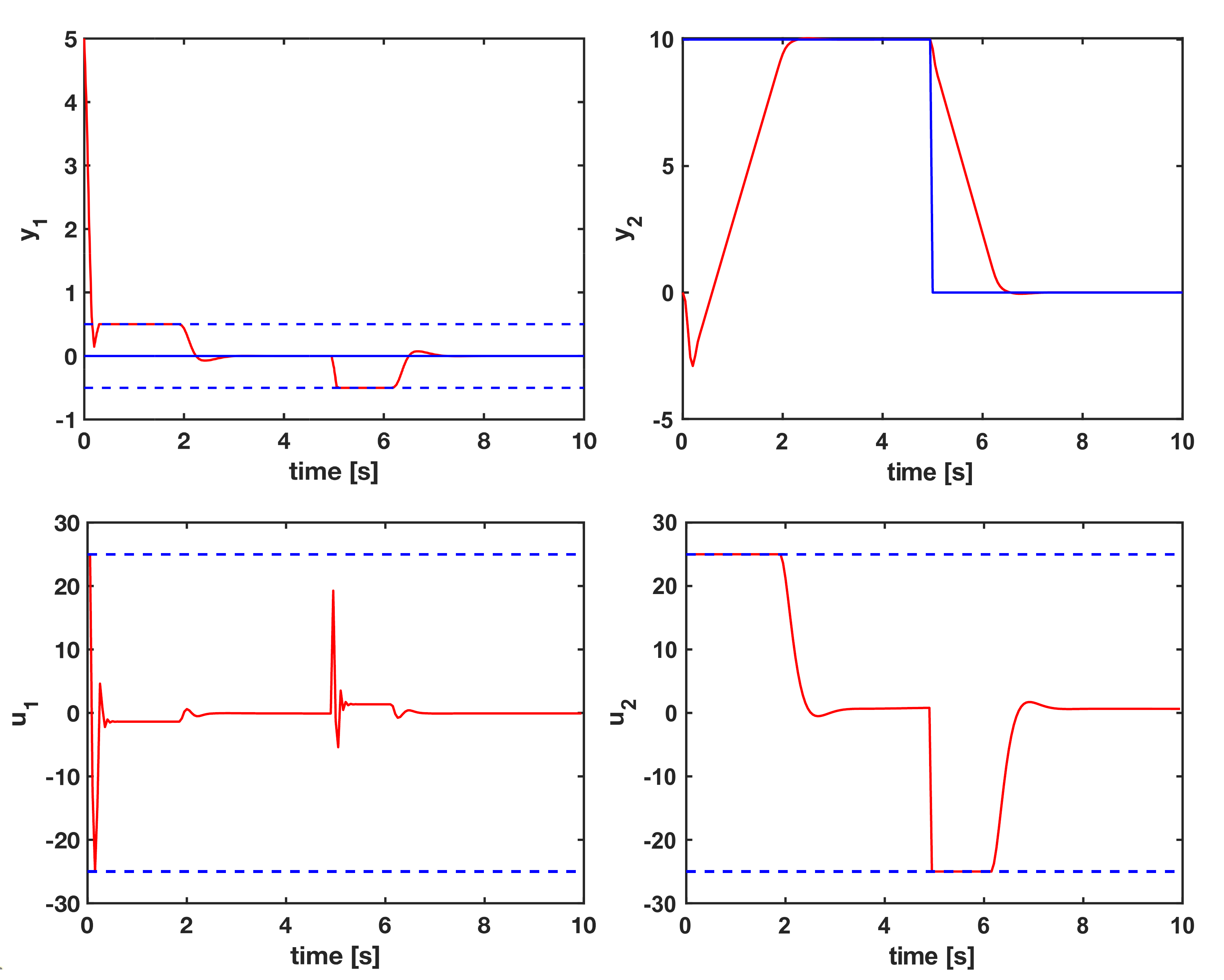}

\vspace{-0.25cm}
    
    \caption{AFTI-F16 closed-loop performance by adopting $\ell_1$-penalty soft-constrained MPC (the values of $\rho$ for output and input constraints are $10^3$ and $10^4$, respectively) using Algorithm \ref{alg_n_3}.}
    \label{fig_AFTI}
\end{figure}
The closed-loop performance is illustrated in Fig.\  \ref{fig_AFTI}, demonstrating that rapidly returning the attack angle $y_1$ to the safe region is treated as the top priority. The pitch angle $y_2$ initially moves in the opposite direction of the desired tracking signal and only begins to track it once $y_1$ has entered the safe region. The physical input constraints are never violated, thanks to the larger value of $\rho$ for input constraints and the sparsity-promoting nature of of $\ell_1$-penalty soft-constrained scheme. 

\subsection{Robust Kalman Filter: Online Lasso}
The Kalman Filter (KF) is a widely used method for state estimation in linear dynamical systems with independently and identically distributed (IID) Gaussian process and measurement noise. However, the KF's performance degrades significantly when handling an additional measurement noise term that is sparse, such as from unknown sensor failures, measurement outliers, or intentional jamming. In \cite{mattingley2010real}, a robust KF (RKF) was designed to handle additive sparse measurement noise. Specifically, consider the linear discrete-time system
\[
x_{t+1}=Ax_t+Bu_t + w_t,\quad y_t = Cx_t + v_t + z_t,
\]
where $x_t\in\mathbb{R}^{n_x}$ is the state to be estimated and $y_t\in\mathbb{R}^{n_y}$ is the available measurement at time step $t$; $w_t\in\mathbb{R}^{n_x}, v_t\in\mathbb{R}^{n_y}$ are the state and measurement Gaussian noise at time step $t$, which are IID $\mathcal{N}(0,Q)$ and $\mathcal{N}(0,R)$, respectively; and $z_t\in\mathbb{R}^m$ is an additional sparse outlier noise at time step $t$. The computing procedure of RKF in \cite{mattingley2010real} is expressed as
\begin{align}
\mathrm{Predict:} \quad \hat{x}_{t \mid t-1} & =A \hat{x}_{t-1 \mid t-1} + Bu_{t-1}, \\
P_{t \mid t-1} & =A P_{t-1 \mid t-1} A^\top + Q . \\
L & =P_{t \mid t-1} C^\top\left(C P_{t \mid t-1} C^\top+R\right)^{-1} \\
\mathrm{Update:} \qquad\ \   e_t & =y_t-C \hat{x}_{t \mid t-1}, \\
\hat{x}_{t \mid t} & =\hat{x}_{t \mid t-1}+L\left(e_t-\hat{z}_t\right), \\
P_{t \mid t} & =(I-L C) P_{t \mid t-1} .
\end{align}
where $P$ is a covariance matrix of the state estimation error and $\hat{z}_t$ is the solution of the Lasso problem,
\begin{equation}\label{problem_RKF_Lasso}
\min_{\hat{z}_t}\ (e_t-\hat{z}_t)^\top W (e_t-\hat{z}_t)+\rho \|\hat{z}_t\|_1
\end{equation} 
where $W=(I-CL)^\top R^{-1}(I-CL)+L^\top P_{t|t-1}^{-1}L$ and $\rho\geq0$ is a regularization parameter. When $\rho$ is sufficiently large, the solution of the Lasso \eqref{problem_RKF_Lasso} yields $\hat{z}_t=0$, causing the RKF estimates to coincide with the standard KF solution.

Compared to KF, RKF achieves better performance at the cost of solving the Lasso \eqref{problem_RKF_Lasso} online, whose execution time should be fast and certified to be smaller than the feedback sampling time. Subsection \ref{sec_Lasso_BoxQP} presents how to transform a Lasso problem into a Box-QP, and therefore our proposed Algorithm \ref{alg_n_3} with certified $O(n^3)$ time complexity can offer an execution time certificate in RKF applications.

Consider a three-tank system as an application example, whose system dynamic matrices are taken from \cite{forces_low_level_estimation},
\[
A = \! \left[\begin{array}{ccc}
    0.9 & 0 & 0 \\
    0 & 0.5 & 0 \\
    0.1 & 0.5 & 0.8
\end{array}\right]\!\!, \ B = \!\left[\begin{array}{c}
   0.5 \\
   0.5 \\
   0
\end{array}\right]\!\!, \ C = \!\left[\begin{array}{ccc}
     1 & 0 & 0 \\
     0 & 0 & 1
\end{array}\right]\!.
\]
The resulting Box-QP problem has the dimension $n=3$ and the optimality level is chosen as $\epsilon=10^{-6}$. Theorem \ref{theorem_n_3_iteration} indicates that the required number of iterations is $\mathcal{N}_\mathrm{iter}=304$ and Theorem \ref{theorem_num_rank} indicates that the required number of rank-1 updates is $\mathcal{N}_{\mathrm{rank}-1}=550$, which can certify the execution time in advance. For example, our computing platform takes $0.04$ ms (25 kHz). Fig.\  {\ref{Fig_RKF_three_tank} plots the relative magnitude square (RMS) error of the RKF and the standard KF estimations, and the magnitude of the RKF is reduced by around $70.06\%, 11.67\%, 57.71\%$ for the states $1,2,3$, respectively.
\begin{figure}[!htbp]
    \centering
    \includegraphics[width=1.0\linewidth]{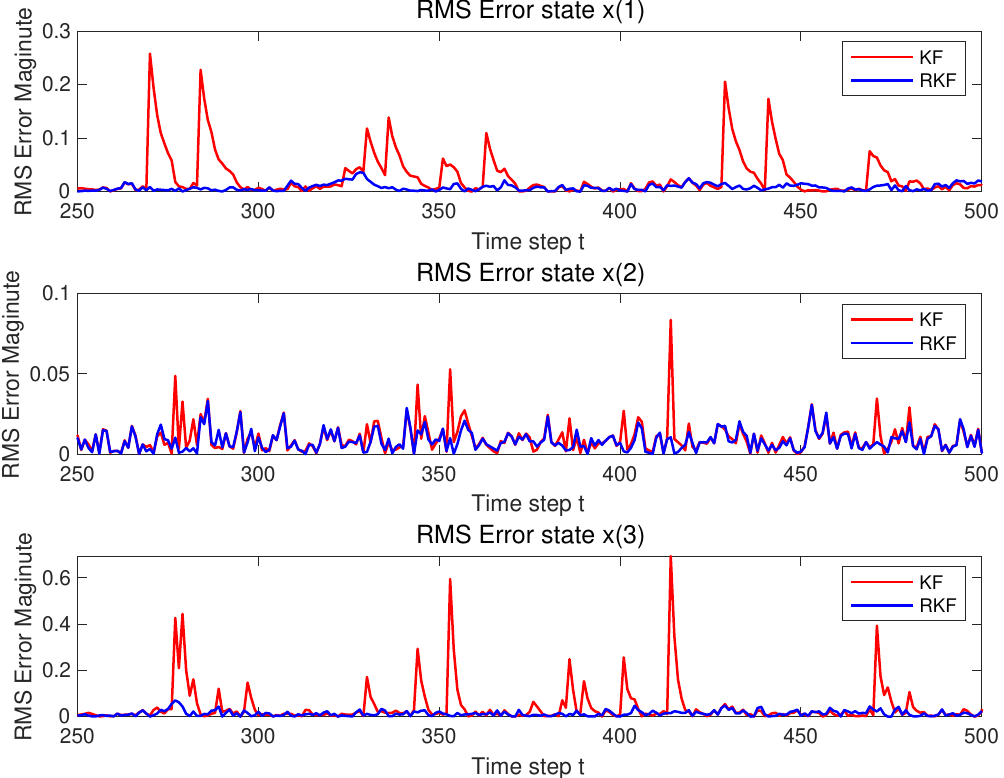}
    \caption{RMS error of the RKF and the standard KF estimations}\label{Fig_RKF_three_tank} 
\end{figure}

\section{Conclusion}
This article investigates the fundamental question of whether QPs can be solved with \textit{data-independent} $O(n^3)$ on par with direct methods, such as Cholesky, LU, and QR factorizations for solving linear systems. To this end, we for the first time propose an \textit{implementable} Box-QP algorithm with $O(n^3)$ time complexity. We demonstrate that Box-QPs have broad applicability, encompassing general strictly convex QPs, Lasso problems, and SVMs. Regarding the proposed Box-QP algorithm, we derive the explicit equation of the upper bound of its required number of iterations and the number of rank-1 updates, thereby being able to offer an execution time certificate for real-time QP-based applications. Moreover, the proposed Box-QP algorithm is \textit{simple to implement} and \textit{matrix-free} in the iterative procedures, making it suitable for embedded deployment.

 
\section{Acknowledgement}
This research was supported by the U.S. Food and Drug Administration under the FDA BAA-22-00123 program, Award Number 75F40122C00200.

\appendix
\subsection{Proof of Theorem \ref{theorem_n_3_iteration}}\label{sec_theorem_sqrt_n}
 Before proving Theorem \ref{theorem_n_3_iteration}, we first prove five lemmas.

\begin{lemma}\label{lemma_scalar_rho}
Under Assumptions \eqref{eqn_assumption_cen}, \eqref{eqn_tilde_x_definition}, and (\ref{eqn_tilde_s_definition}) of Theorem \ref{theorem_n_3_iteration}, it holds that
\begin{equation}
    \|(\Tilde{x}\pdot\Tilde{s})^{{1}/{2}}\|_\infty\leq(1+\delta) (1+\alpha)^{{1}/{2}}\tau^{{1}/{2}}
\end{equation}
and
\begin{equation}
 \|(\Tilde{x}\pdot\Tilde{s})^{-{1}/{2}}\|_\infty\leq(1+\delta)(1-\alpha)^{-{1}/{2}}\tau^{-{1}/{2}}.
\end{equation}
\end{lemma}
\begin{proof}
   Assumption \eqref{eqn_assumption_cen} implies that
\begin{equation}
    \|x\pdot s-\tau e\|_\infty\leq \|x\pdot s-\tau e\|\leq\alpha\tau.
\end{equation}
That is, for every $i=1,\cdots{},2n$,
\begin{equation}
    (1-\alpha)\tau \leq x_i s_i \leq (1+\alpha) \tau.
\end{equation}
For every $i=1,\cdots{},2n$,
\begin{equation}  (x_is_i)^{{1}/{2}}\leq(1+\alpha)^{{1}/{2}}\tau^{{1}/{2}},
\end{equation}
    and
\begin{equation}
     (x_is_i)^{-{1}/{2}}\leq(1-\alpha)^{-{1}/{2}}\tau^{-{1}/{2}};
\end{equation}
    that is,
 \begin{equation}
    \|(x\pdot s)^{{1}/{2}}\|_\infty\leq(1+\alpha)^{{1}/{2}}\tau^{{1}/{2}},
\end{equation}
\begin{equation}
\|(x\pdot s)^{-{1}/{2}}\|_\infty\leq(1-\alpha)^{-{1}/{2}}\tau^{-{1}/{2}}.
\end{equation}
Assumptions (\ref{eqn_tilde_x_definition}) and (\ref{eqn_tilde_s_definition}) imply that 
\begin{equation}
\left\| \left( \frac{\tilde{x}}{x}\right)^{\!\!{1}/{2}} \right\|_\infty \leq(1+\delta)^{{1}/{2}},\quad \left\|\left(\frac{\tilde{x}}{x}\right)^{\!\!-{1}/{2}}\right\|_\infty \leq(1+\delta)^{{1}/{2}},
\end{equation}
\begin{equation}
\left\|\left(\frac{\tilde{s}}{s}\right)^{\!\!{1}/{2}}\right\|_\infty\leq(1+\delta)^{{1}/{2}}, \quad \left\|\left(\frac{\tilde{s}}{s}\right)^{\!\!{1}/{2}}\right\|_\infty\leq(1+\delta)^{{1}/{2}};
\end{equation}
thus,
\begin{equation}
    \begin{aligned}
    \|(\tilde{x} \pdot \tilde{s})^{{1}/{2}}\|_\infty &= \left\| \left( \left( \frac{\tilde{x}}{x}\right) \pdot \left( \frac{\tilde{s}}{s}\right) \pdot(x\pdot s) \right)^{\!\!{1}/{2}}\right\|_\infty\\
&\leq \left\| \left( \frac{\tilde{x}}{x}\right)^{\!\!{1}/{2}} \right\|_\infty  \left\| \left( \frac{\tilde{s}}{s}\right)^{\!\!{1}/{2}} \right\|_\infty  \left\|(x\pdot s)^{{1}/{2}}\right\|_\infty\\
&\leq(1+\delta) (1+\alpha)^{{1}/{2}}\tau^{{1}/{2}}
\end{aligned}
\end{equation}
and similarly, 
\begin{equation}
    \|(\tilde{x} \pdot \tilde{s})^{-{1}/{2}}\|_\infty\leq(1+\delta)(1-\alpha)^{-{1}/{2}}\tau^{-{1}/{2}},
  \end{equation}
    which completes the proof.
\end{proof}

\begin{lemma}\label{lemma_dx_ds_u}
Let $(\Delta z, \Delta x, \Delta s)$ be the solutions of \eqref{eqn_approximate_Newton_step} and denote
\begin{equation}
    \begin{aligned}
       &r\triangleq(\tilde{x}\pdot\tilde{s})^{-{1}/{2}} \pdot(\tau e- x\pdot s),\\
    &p_r=\sqrt{\frac{\tilde{s}}{\tilde{x}}}\Delta x,\quad q_r=\sqrt{\frac{\tilde{x}} {\tilde{s}}}\Delta s.
    \end{aligned}
\end{equation}
Under the assumptions of Theorem \ref{theorem_n_3_iteration}, it holds that
\begin{subequations}
\begin{align}
&\|r\|\leq(1+\delta)\alpha(1-\alpha)^{-{1}/{2}}\tau^{{1}/{2}},\label{eqn_p_q_r_a}\\
&\|p_r\|\leq\|r\|,\quad \|q_r\|\leq\|r\|,\label{eqn_p_q_r_b}\\
    &\|\Delta x \pdot \Delta s\|\leq \frac{\|r\|^2}{2},\label{eqn_p_q_r_c}\\
    &~\|p_r\| + \|q_r\|\leq\sqrt{2}\|r\|,\label{eqn_p_q_r_d}\\
&\left\|\frac{\Delta x}{x}\right\|\leq\eta,\quad \left\|\frac{\Delta s}{s}\right\|\leq\eta\label{eqn_p_q_r_e},
\end{align}
\end{subequations}
where $\eta\triangleq\frac{(1+\delta)^3\alpha}{1-\alpha}$.
\end{lemma}

\begin{proof}
By the definition of $r$, we have that
\begin{equation}
\begin{aligned}
    \|r\|&=\left\|(\tilde{x} \pdot \tilde{s})^{-{1}/{2}} \pdot (\tau e-x \pdot s) \right\|\\   
&\leq\|(\tilde{x} \pdot \tilde{s})^{-{1}/{2}}\|_\infty  \|x\pdot s-\tau e\|\\
&\quad\quad \text{(by Assumption \eqref{eqn_assumption_cen} of Theorem \ref{theorem_n_3_iteration})}\\
&\leq \|(\tilde{x} \pdot \tilde{s})^{-{1}/{2}}\|_\infty \cdot \alpha\tau  \qquad \text{(by Lemma \ref{lemma_scalar_rho}})\\
        &\leq (1+\delta) \alpha (1-\alpha)^{-{1}/{2}} \tau^{{1}/{2}}
    \end{aligned}
    \end{equation}
    which completes the proof of \eqref{eqn_p_q_r_a}.
    
Multiplying \eqref{eqn_approximate_c}  by $(\tilde{x}\pdot\tilde{s})^{-{1}/{2}}$ gives that
\begin{equation}
    p_r + q_r = (\tilde{x} \pdot \tilde{s})^{-{1}/{2}} \pdot (\tau e - x\pdot s) = r
\end{equation}
and, by \eqref{eqn_Delta_x_Delta_s_positive}, we have that
\[
    p_r^\top q_r = \Delta x^\top \Delta s\geq0.
\]
Then $\|r\|^2=\|p_r+q_r\|^2=\|p_r\|^2+2p_r^\top q_r + \|q_r\|^2\geq\|p_r\|^2+\|q_r\|^2$, which implies that $\|p_r\|\leq\|r\|$ and $\|q_r\|\leq\|r\|$, completing the proof of \eqref{eqn_p_q_r_b}.

Since $(\|p_r\|-\|q_r\|)^2 = \|p_r\|^2 - 2\|p_r\| \|q_r\| + \|q_r\|^2 \geq 0$, it holds that
\begin{equation}
\|p_r\| \|q_r\|\leq\frac{\|p_r\|^2+\|q_r\|^2}{2}\leq\frac{\|r\|^2}{2}.
\end{equation}
Thus,
\begin{equation}
       \|\Delta x \pdot \Delta s\| = \|p_r \pdot q_r\| \leq \|p_r\| \|q_r\| \leq \frac{\|r\|^2}{2},
\end{equation}
which completes the proof of \eqref{eqn_p_q_r_c}.
    
By $\|p_r\|^2+\|q_r\|^2\leq\|r\|^2$ and $\|p_r\| \|q_r\|\leq\frac{\|r\|^2}{2}$, we have that
\begin{equation}
\left(\|p_r\|+\|q_r\|\right)^2=\|p_r\|^2+\|q_r\|^2+2\|p_r\| \|q_r\| \leq 2\|r\|^2,
\end{equation}
That is,
$\|p_r\|+\|q_r\|\leq\sqrt{2}\|r\|$, which completes the proof of \eqref{eqn_p_q_r_d}.
    
Furthermore, we have that
\begin{equation}
\begin{aligned}
\left\|\frac{\Delta x}{x} \right\|&= \left \| \frac{\tilde{x}}{x} \frac{1}{\sqrt{\tilde{x} \pdot \tilde{s}}} \sqrt{\frac{\tilde{s}}{\tilde{x}}} \Delta x \right\| \\
&\leq\left\|\frac{\Tilde{x}}{x}\right\|_\infty  \left\|(\Tilde{x}\pdot\Tilde{s})^{-{1}/{2}} \right\|_\infty  \left\|p\right\|\\
&\quad\quad \text{(by Assumption (\ref{eqn_tilde_x_definition}), Lemma \ref{lemma_scalar_rho}, and  \eqref{eqn_p_q_r_b})} \\
&\leq(1+\delta)^2 (1-\alpha)^{-{1}/{2}} \tau^{-{1}/{2}}\|r\| \qquad \text{(by \eqref{eqn_p_q_r_a})} \\
&\leq\frac{(1+\delta)^3\alpha}{1-\alpha}=\eta.
\end{aligned}
\end{equation}
Similarly we have that $\left\|\frac{\Delta s}{s}\right\|\leq\eta$, which completes the proof of \eqref{eqn_p_q_r_e}.
\end{proof}

\begin{lemma}\label{lemma_bar_x_bar_s}
    Given the next iterate $(z^+,x^+,s^+)$ from \eqref{eqn_next_point}
    and under the assumptions of Theorem \ref{theorem_n_3_iteration}, then 
\begin{equation}
        \|x^+\pdot s^+-\tau e\|\leq\sigma\tau
\end{equation}
    where 
\begin{equation}
\sigma=\sqrt{2}\delta(1+\delta)^2\alpha\sqrt{\frac{1+\alpha}{1-\alpha}}+\frac{(1+\delta)^2\alpha^2}{2(1-\alpha)}.
\end{equation}
\end{lemma}
\begin{proof}
First, by \eqref{eqn_next_point} and \eqref{eqn_approximate_c}, we have that
    \[
    \begin{aligned}
        &~x^+\pdot s^+ = x\pdot s + (s\pdot\Delta x+x\pdot\Delta s) + \Delta x\pdot \Delta s \\
        &= x\pdot s + (\Tilde{s}\pdot\Delta x + \Tilde{x}\pdot\Delta s) + (s-\Tilde{s})\pdot\Delta x + (x-\Tilde{x})\pdot\Delta s\\
        &\quad +\Delta x\Delta s\\
        &=\tau e + (s-\Tilde{s})\pdot\Delta x + (x-\Tilde{x})\pdot\Delta s+\Delta x\pdot\Delta s.
    \end{aligned}
    \]
Then we have that
    \[
    \begin{aligned}
        &\quad\|x^+\pdot s^+-\tau e\|\\
        &\leq\|(s-\Tilde{s})\pdot\Delta x + (x-\Tilde{x})\pdot\Delta s\| + \|\Delta x\pdot\Delta s\|\\
        &=\left\|\sqrt{\Tilde{x}\pdot\Tilde{s}}\left(\frac{s-\Tilde{s}}{\Tilde{s}}\sqrt{\frac{\Tilde{s}}{\Tilde{x}}}\Delta x + \frac{x-\Tilde{x}}{\Tilde{x}}\sqrt{\frac{\Tilde{x}}{\tilde{s}}}\Delta s\right) \right\| + \|\Delta x\pdot\Delta s\|\\
        &\leq\left\|\sqrt{\tilde{x}\pdot\tilde{s}}\right\|_\infty \left(\left\|\frac{s-\tilde{s}}{\tilde{s}}\right\|_\infty \left\|p_r\right\|  + \left\|\frac{x-\tilde{x}}{\tilde{x}}\right\|_\infty \left\|q_r\right\| \right)  \\
        &\quad + \|\Delta x\pdot\Delta s\|\\
        &\quad\text{(by Lemma \ref{lemma_scalar_rho}, Assumptions (\ref{eqn_tilde_x_definition}), (\ref{eqn_tilde_s_definition}), and  \eqref{eqn_p_q_r_b})}\\
    &\leq (1+\delta)(1+\alpha)^{{1}/{2}}\tau^{{1}/{2}}\left(\delta \left \|p_r\right\| + \delta\left\|q_r \right\|\right) + \frac{\|r\|^2}{2}\\
&\quad\text{(by \eqref{eqn_p_q_r_d})} \\
&\leq(1+\delta)(1+\alpha)^{{1}/{2}}\tau^{{1}/{2}} \delta \sqrt{2} \|r\|+ \frac{\|r\|^2}{2}\\
 &\quad\text{(by \eqref{eqn_p_q_r_a})}\\
&\leq\Bigg(\sqrt{2}\delta(1+\delta)^2\alpha\sqrt{\frac{1+\alpha}{1-\alpha}}+\frac{(1+\delta)^2\alpha^2}{2(1-\alpha)}\Bigg)\tau\\
&=\sigma\tau,
    \end{aligned}
    \]
    which completes the proof.
\end{proof}

Now, we are ready to provide the proof of Theorem \ref{theorem_n_3_iteration}.

\textbf{(1) Proof of \eqref{eqn_next_iterate_center}:} according to the choice of the approximation parameter $\delta$ in  \eqref{eqn_alpha_delta_beta} and by  \eqref{eqn_p_q_r_e}. we have that
\begin{equation}
\begin{aligned}
&\left\|\frac{\Delta x}{x}\right\|_\infty \leq \left\|\frac{\Delta x}{x}\right\|\leq\frac{(1+\delta)^3\alpha}{1-\alpha}<1,\\
&\left\|\frac{\Delta s}{s}\right\|_\infty \leq \left\|\frac{\Delta s}{s}\right\|\leq\frac{(1+\delta)^3\alpha}{1-\alpha}<1.
\end{aligned}
\end{equation}
That is, based on $(x,s)>0$, we have that
\begin{equation}
x^+=x+\Delta x>0,\quad s^+=s+\Delta s>0.
\end{equation}
Furthermore, by Lemma \ref{lemma_bar_x_bar_s} and Assumption \eqref{eqn_assumption_tau}, we have that
\begin{equation}\label{eqn_next_x_s_mu}
\begin{aligned}
\xi^+&\triangleq\frac{\left\|x^+\pdot s^+-\tau^+e\right\|}{\tau^+}\\
&=\frac{\left\|x^+\pdot s^+-\tau e + \tau e - \tau^+e \right\|}{\tau^+}\\
&\leq\frac{\left\| x^+\pdot s^+-\tau e\right\| + \left\|\tau^+e - \tau e\right\|}{\tau^+}\\
    &=\frac{\left\| x^+\pdot s^+-\tau e\right\| + \sqrt{2n}\left|\tau^+-\tau\right|}{\tau^+}\\
    &\quad\text{(by Lemma \ref{lemma_bar_x_bar_s} and  \eqref{eqn_assumption_tau})}\\
&\leq\frac{\sigma+\beta}{1-\frac{\beta}{\sqrt{2n}}}.
\end{aligned}    
\end{equation}
To make $\frac{\sigma+\beta}{1-\frac{\beta}{\sqrt{2n}}}=\alpha$, we choose $ \beta=\frac{\alpha-\sigma}{1+\frac{\alpha}{\sqrt{2n}}}>0$, which is  \eqref{eqn_alpha_delta_beta}, and we have that
\[
\xi^+\leq\alpha.
\]
which completes the proof of  \eqref{eqn_next_iterate_center} of Theorem \ref{theorem_n_3_iteration}. 

\textbf{(2) Proof of \eqref{eqn_number_iterations}:} let $\tau^k$ and $\tau^0$ denote the $k$th and initial value of $\tau$ ($\tau^0=1$), and according to  Assumption \eqref{eqn_assumption_tau}, we have that
\[
\tau^k=\left(1-\frac{\beta}{\sqrt{2n}}\right)^{\!\!k}\tau^0=\left(1-\frac{\beta}{\sqrt{2n}}\right)^{\!\!k}
\]
Let $(x^k,s^k)$ be the $k$th iterate
then the duality gap $(x^k)^\top s^k$ has
\[
\begin{aligned}
(x^k)^\top s^k&= e^\top (x^k\pdot s^k) = e^\top \left(\tau^k e+(x^k\pdot s^k-\tau^k e) \right)\\
&=2n\tau^k + e^\top (x^k\pdot s^k-\tau^k e).    
\end{aligned}
\]
By the Cauchy-Schwarz inequality, we have that
\[
|e^\top (x^k\pdot s^k-\tau^k e) |\leq\|e\|_2 \|x^k\pdot s^k-\tau^k e\|_2=\sqrt{2n} \|x^k\pdot s^k-\tau ^ke\|.
\]
Then, by the proof \eqref{eqn_next_iterate_center}, all iterates are positive and lie in the narrow neighborhood: $\|x^k\pdot s^k-\tau ^ke\|\leq\alpha\tau^k$. We have that
\[
2n\tau^k-\sqrt{2n}\alpha\tau^k\leq(x^k)^\top s^k\leq 2n\tau^k+\sqrt{2n}\alpha\tau^k,
\]
namely, that
\begin{equation}\label{eqn_iteration_inequality}
\begin{aligned}
&(2n-\alpha\sqrt{2n})\left(1-\frac{\beta}{\sqrt{2n}}\right)^{\!\!k}\leq(x^{k})^\top s^{k}\\
&\quad\quad\leq(2n+\alpha\sqrt{2n})\left(1-\frac{\beta}{\sqrt{2n}}\right)^{\!\!k}.  
\end{aligned}
\end{equation}
Hence, $(x^k)^\top s^k\leq\epsilon$ holds if
\begin{equation}
(2n+\alpha\sqrt{2n})\left(1-\frac{\beta}{\sqrt{2n}}\right)^{\!\!k}\leq\epsilon.
\end{equation}
Taking the logarithm of both sides gives that
\begin{equation}
k\log\!\left(1-\frac{\beta}{\sqrt{2n}} \right)+\log\!\left(2n+\alpha\sqrt{2n}\right)\leq\log(\epsilon),
\end{equation}
which holds if 
\begin{equation}
k\geq\frac{\log\left(\frac{\epsilon}{2n+\alpha\sqrt{2n}}\right)}{\log\left(1-\frac{\beta}{\sqrt{2n}}\right)}=\frac{\log\left(\frac{2n+\alpha\sqrt{2n}}{\epsilon}\right)}{-\log\left(1-\frac{\beta}{\sqrt{2n}}\right)}
\end{equation}
Thus, after at most
\[
\mathcal{N}_\mathrm{iter}=\!\left\lceil\frac{\log\left(\frac{2n+\alpha\sqrt{2n}}{\epsilon}\right)}{-\log\left(1-\frac{\beta}{\sqrt{2n}}\right)}\right\rceil
\]
iterations, $(x^0)^\top s^0\leq\epsilon$.

Furthermore, the above worst-case iteration complexity can be proved to be \textbf{exact} by using \eqref{eqn_iteration_inequality}. For the $(k-1)$th iteration,  we have that
\begin{equation}
\begin{aligned}
&(2n-\alpha\sqrt{2n})\left(1-\frac{\beta}{\sqrt{2n}}\right)^{\!\!k-1}\leq(x^{k-1})^\top s^{k-1}\\
&\quad\quad\leq(2n+\alpha\sqrt{2n})\left(1-\frac{\beta}{\sqrt{2n}}\right)^{\!\!k-1}.  
\end{aligned}
\end{equation}
The duality gap will reach the optimality level $\epsilon$ no earlier than and no later than $\mathcal{N}_\mathrm{iter}$ iterations if
\[
(2n-\alpha\sqrt{2n})\left(1-\frac{\beta}{\sqrt{2n}}\right)^{\!\!k}\leq(2n+\alpha\sqrt{2n})\left(1-\frac{\beta}{\sqrt{2n}}\right)^{\!\!k-1},
\]
that is,  
\[
(2n-\alpha\sqrt{2n})\left(1-\frac{\beta}{\sqrt{2n}}\right)\leq2n+\alpha\sqrt{2n}, \ \forall n=1,2,\cdots
\]
Substituting $\beta=\frac{\alpha-\sigma}{1+\frac{\alpha}{\sqrt{2n}}}$ into the above inequality, we need to prove that
\begin{equation}
(\sqrt{2n}-\alpha)(\sqrt{2n}+\sigma)\leq(\sqrt{2n}+\alpha)^2,
\end{equation}
that is, $\sigma\sqrt{2n}\leq3\alpha\sqrt{2n}+\alpha^2+\alpha\sigma$, which is always true as $\alpha^2+\alpha\sigma>0$ and we can prove $\sigma\leq3\alpha$ by $\sigma=\sqrt{2}\delta(1+\delta)^2\alpha\sqrt{\frac{1+\alpha}{1-\alpha}}+\frac{(1+\delta)^2\alpha^2}{2(1-\alpha)}$ with $\alpha=0.3, \delta=0.15$. This completes the proof of \eqref{eqn_number_iterations} of Theorem \ref{theorem_n_3_iteration}.

\subsection{Proof of Corollary \ref{corollary_n_35}}\label{sec_corollary_n_35}
\begin{proof}
Corollary \ref{corollary_n_35} is a special case of Theorem \ref{theorem_n_3_iteration} when $\delta=0$, thus their proof is similar. Taking $\delta=0$ can easily prove that
\begin{itemize}
\item[1)] \[
\left\|(x\pdot s)^{{1}/{2}} \right\|_\infty \leq (1-\alpha)^{-{1}/{2}} \tau^{-{1}/{2}}
        \]
\item[2)] Denote $v\triangleq(x\pdot s)^{-{1}/{2}} \pdot (\tau e-x\pdot s)$, then 
\[
\|v\|\leq\alpha(1-\alpha)^{-{1}/{2}}\tau^{{1}/{2}}
\]
\item[3)] Denote $p_v\triangleq\sqrt{\frac{s}{x}}\Delta x,~q_v\triangleq \sqrt{\frac{x}{s}}\Delta s$, then
\[
\|p_v\|\leq\|v\|,\quad \|q_v\|\leq\|v\|.
\]
\item[4)] \[
        \|\Delta x\pdot\Delta s\|\leq \frac{\|v\|^2}{2}
        \]
\item[5)]\[
        \|p_r\| + \|q_r\|\leq\sqrt{2}\|v\|
        \]
\item[6)] \[        \left\|\frac{\Delta x}{x}\right\|\leq\frac{\alpha}{1-\alpha},\quad \left\|\frac{\Delta s}{s}\right\|\leq\frac{\alpha}{1-\alpha}.
        \]
    \end{itemize}
Then we have that
    \[
    \begin{aligned}
\xi^+&=\frac{\left\|x^+\pdot s^+-\tau^+e \right\|}{\tau^+}\\ &=\frac{\left\|x\pdot s+x\pdot\Delta s+s\pdot\Delta x+\Delta x\pdot\Delta s-\tau^+e \right\|}{\tau^+}\\
&=\frac{\left\|\tau e+\Delta x\pdot\Delta s-\tau^+e \right\|}{\tau^+}\\
&\leq\frac{\|\Delta x\pdot\Delta s\|}{\tau^+} + \frac{\sqrt{2n}|\tau-\tau^+|}{\tau^+}\\
        &\leq \frac{\|v\|^2}{2\tau^+}+ \frac{\sqrt{2n}|\tau-\tau^+|}{\tau^+}\\
&\leq\frac{\sigma+\beta}{1-\frac{\beta}{\sqrt{2n}}},
\end{aligned}
    \]
where $\sigma=\frac{\alpha^2}{2(1-\alpha)}$. To make $\frac{\sigma+\beta}{1-\frac{\beta}{\sqrt{2n}}}=\alpha$, we choose $ \beta=\frac{\alpha-\sigma}{1+\frac{\alpha}{\sqrt{2n}}}>0.$
    
    The remaining proof is similar to the \textbf{Proof of \eqref{eqn_number_iterations}} in Appendix \ref{sec_theorem_sqrt_n}.
\end{proof}

\subsection{Proof of Theorem \ref{theorem_num_rank}}\label{sec_theorem_num_rank}
For the $k$th iteration of Algorithm \ref{alg_n_3}, the total number of rank-1 updates is bounded by 
\[
\sum_{k=1}^{\mathcal{N}} |I_{\gamma}^k\cup I_{\theta}^k\cup I_{\phi}^k\cup I_{\psi}^k|\leq\sum_{k=1}^{\mathcal{N}_\mathrm{iter}}I_{\gamma}^k+\sum_{k=1}^{\mathcal{N}_\mathrm{iter}} I_{\theta}^k+\sum_{k=1}^{\mathcal{N}_\mathrm{iter}} I_{\phi}^k+\sum_{k=1}^{\mathcal{N}_\mathrm{iter}} I_{\psi}^k
\]
where $\mathcal{N}_\mathrm{iter}$ is the total number of iterations from  \eqref{eqn_number_iterations}.

\begin{definition}
    For every element $i=1,\cdots{},n$ in $\gamma,\theta,\phi,\psi$, define the sets $\mathcal{K}_{\gamma_i},\mathcal{K}_{\theta_i},\mathcal{K}_{\phi_i},\mathcal{K}_{\psi_i}$, the subsequences of $\{1,\cdots{},\mathcal{N}_\mathrm{iter}\}$, such that:
    \[
    \begin{aligned}
    &\text{Let } \mathcal{K}_{\gamma_i}=\{1\}, \text{and } k\in\{2,\cdots{},\mathcal{N}_\mathrm{iter}\} \text{ to } \mathcal{K}_{\gamma_i} \text{ iff } i\in I_{\gamma}^k,\\
    &\text{Let } \mathcal{K}_{\theta_i}=\{1\}, \text{and } k\in\{2,\cdots{},\mathcal{N}_\mathrm{iter}\} \text{ to } \mathcal{K}_{\theta_i} \text{ iff } i\in I_{\theta}^k,\\
    &\text{Let } \mathcal{K}_{\phi_i}=\{1\}, \text{and } k\in\{2,\cdots{},\mathcal{N}_\mathrm{iter}\} \text{ to } \mathcal{K}_{\phi_i} \text{ iff } i\in I_{\phi}^k,\\
    &\text{Let } \mathcal{K}_{\psi_i}=\{1\}, \text{and } k\in\{2,\cdots{},\mathcal{N}_\mathrm{iter}\} \text{ to } \mathcal{K}_{\psi_i} \text{ iff } i\in I_{\psi}^k.
    \end{aligned}
    \]
\end{definition}
\begin{remark}
    By the definition of the sets $\mathcal{K}_{\gamma_i},\mathcal{K}_{\theta_i},\mathcal{K}_{\phi_i},\mathcal{K}_{\psi_i}$, summarizing the cardinality of the set $\mathcal{K}_{\gamma_i},\mathcal{K}_{\theta_i},\mathcal{K}_{\phi_i},\mathcal{K}_{\psi_i}$ for all elements $i=1,\cdots{},n$ in $\gamma,\theta,\phi,\psi$, we have that \begin{equation}\label{eqn_I_x}
    \begin{aligned}
&\sum_{i=1}^{n}\left|\mathcal{K}_{\gamma_i} \right|-n=\sum_{k=1}^{\mathcal{N}_\mathrm{iter}}\left|I_{\gamma}^k \right|,~\sum_{i=1}^{n}\left|\mathcal{K}_{\theta_i} \right|-n=\sum_{k=1}^{\mathcal{N}_\mathrm{iter}}\left|I_{\theta}^k \right|\\
&\sum_{i=1}^{n}\left|\mathcal{K}_{\phi_i} \right|-n=\sum_{k=1}^{\mathcal{N}_\mathrm{iter}}\left|I_{\phi}^k \right|,~\sum_{i=1}^{n}\left|\mathcal{K}_{\psi_i} \right|-n=\sum_{k=1}^{\mathcal{N}_\mathrm{iter}}\left|I_{\psi}^k \right| 
    \end{aligned}
    \end{equation}
    where the term ``$-n$" comes from the fact that, at the first iteration ($k=1$): $I_\gamma^1=\emptyset, I_{\theta}^1=\emptyset, I_{\phi}^1=\emptyset, I_{\psi}^1=\emptyset$ and all $\mathcal{K}_{\gamma_i},\mathcal{K}_{\theta_i},\mathcal{K}_{\phi_i},\mathcal{K}_{\psi_i}$ ($i=1,\cdots{},n$) already include the first iteration ($k=1$) case.
\end{remark}
\begin{lemma}\label{lemma_j_th_number_rank_1_update}
    For every $i=1,\cdots{},n$, we have that
\begin{subequations}
        \begin{align}
&|\mathcal{K}_{\gamma_i} |-1 \leq\frac{\sum_{k=1}^{\mathcal{N}_\mathrm{iter}-1}\left(\frac{|\gamma_i^{k+1}-\gamma_i^k|}{\gamma_i^k}\right)}{\left(1-\eta\right)\log(1+\delta)}\label{eqn_K_gamma}\\
&|\mathcal{K}_{\theta_i} |-1 \leq\frac{\sum_{k=1}^{\mathcal{N}_\mathrm{iter}-1}\left(\frac{|\theta_i^{k+1}-\theta_i^k|}{\theta_i^k}\right)}{\left(1-\eta\right)\log(1+\delta)}\label{eqn_K_theta}\\
&|\mathcal{K}_{\phi_i} |-1 \leq\frac{\sum_{k=1}^{\mathcal{N}_\mathrm{iter}-1}\left(\frac{|\phi_i^{k+1}-\phi_i^k|}{\phi_i^k}\right)}{\left(1-\eta\right)\log(1+\delta)}\label{eqn_K_phi}\\
&|\mathcal{K}_{\psi_i} |-1 \leq\frac{\sum_{k=1}^{\mathcal{N}_\mathrm{iter}-1}\left(\frac{|\psi_i^{k+1}-\psi_i^k|}{\psi_i^k}\right)}{\left(1-\eta\right)\log(1+\delta)}\label{eqn_K_psi}        \end{align}
    \end{subequations}
\end{lemma}
\begin{proof}
By \eqref{eqn_p_q_r_e} in Lemma \ref{lemma_dx_ds_u}, we have that
\begin{equation}
\left\|\frac{\Delta x}{x}\right\|_\infty \leq \left\|\frac{\Delta x}{x}\right\|\leq\eta,~\left\|\frac{\Delta s}{s}\right\|_\infty \leq \left\|\frac{\Delta s}{s}\right\|\leq\eta.
\end{equation}
Moreover, $(x,s)$ stays positive during the whole iterations, so for all $i=1,\cdots{},n$, we have that, $\forall k=1,\cdots{},\mathcal{N}_\mathrm{iter}-1$,
\begin{equation}
\begin{aligned}
&\frac{|\gamma_i^{k+1}-\gamma_i^{k}|}{\gamma_i^k}\leq\eta,
\quad \frac{|\theta_i^{k+1}-\theta_i^{k}|}{\theta_i^k}\leq\eta,\\
&\frac{|\phi_i^{k+1}-\phi_i^{k}|}{\phi_i^k}\leq\eta, \quad \frac{|\psi_i^{k+1}-\psi_i^{k}|}{\psi_i^k}\leq\eta.
\end{aligned}
\end{equation}
These inequalities directly lead to, $\forall k=1,\cdots{},\mathcal{N}_\mathrm{iter}-1,$
\begin{equation}
    \begin{aligned}
    &(1-\eta)\gamma_i^k\leq \gamma_i^{k+1}\leq(1+\eta)\gamma_i^k,\\
    &(1-\eta)\theta_i^k\leq \theta_i^{k+1}\leq(1+\eta)\theta_i^k,\\ 
    &(1-\eta)\phi_i^k\leq \phi_i^{k+1}\leq(1+\eta)\phi_i^k,\\ 
    &(1-\eta)\psi_i^k\leq \psi_i^{k+1}\leq(1+\eta)\psi_i^k,
    \end{aligned}
\end{equation}

Thus, for all $i=1,\cdots{},n$, we have that, $\forall k=1,\cdots,\mathcal{N}_\mathrm{iter}-1$,
\begin{equation}\label{eqn_x_j_k}
    \begin{aligned}
    \frac{|\gamma_i^{k+1}-\gamma_i^{k}|}{\gamma_i^{k+1}}&=\frac{\gamma_i^k}{\gamma_i^{k+1}}\frac{|\gamma_i^{k+1}-\gamma_i^{k}|}{\gamma_i^{k}}\leq\frac{1}{1-\eta}\frac{|\gamma_i^{k+1}-\gamma_i^{k}|}{\gamma_i^k},\\
     \frac{|\theta_i^{k+1}-\theta_i^{k}|}{\theta_i^{k+1}}&=\frac{\theta_i^k}{\theta_i^{k+1}}\frac{|\theta_i^{k+1}-\theta_i^{k}|}{\theta_i^{k}}\leq\frac{1}{1-\eta}\frac{|\theta_i^{k+1}-\theta_i^{k}|}{\theta_i^k},\\
    \frac{|\phi_i^{k+1}-\phi_i^{k}|}{\phi_i^{k+1}}&=\frac{\phi_i^k}{\phi_i^{k+1}}\frac{|\phi_i^{k+1}-\phi_i^{k}|}{\phi_i^{k}}\leq\frac{1}{1-\eta}\frac{|\phi_i^{k+1}-\phi_i^{k}|}{\phi_i^k},\\
    \frac{|\psi_i^{k+1}-\psi_i^{k}|}{\psi_i^{k+1}}&=\frac{\psi_i^k}{\psi_i^{k+1}}\frac{|\psi_i^{k+1}-\psi_i^{k}|}{\psi_i^{k}}\leq\frac{1}{1-\eta}\frac{|\psi_i^{k+1}-\psi_i^{k}|}{\psi_i^k},  
    \end{aligned}
\end{equation}

Taking the proof of \eqref{eqn_K_gamma} related to $\mathcal{K}_{\gamma_i}$ as an example, let $k_1$ and $k_2$ be every pair of consecutive indices in the set $\mathcal{K}_{\gamma_i}$. By the definition of the set $\mathcal{K}_{\gamma_i}$, we have that
\begin{equation}
\frac{\gamma_i^{k_2}}{\gamma_i^{k_1}}\notin \left[\frac{1}{1+\delta},1+\delta \right],
\end{equation}
    that is,  $1+\delta<\frac{\gamma_i^{k_2}}{\gamma_i^{k_1}}$. Taking logarithms gives that
\begin{equation}
    \begin{aligned}
        \log(1+\delta)&\leq |\log(\gamma_i^{k_2})-\log(\gamma_i^{k_1})|\\
        &= \left|\int_{\gamma_i^{k_1}}^{\gamma_i^{k_2}}\frac{1}{\gamma}d\gamma \right|\\
        &\leq \sum_{k=k_1}^{k_2-1}\max\left\{ \frac{|\gamma_i^{k+1}-\gamma_i^k|}{\gamma_i^{k+1}}, \frac{|\gamma_i^{k+1}-x_i^k|}{\gamma_i^k}\right\}\\
         &\quad\text{(by inequality \eqref{eqn_x_j_k}) )}\\
        &\leq\sum_{k=k_1}^{k_2-1}\frac{1}{1-\eta}\left(\frac{|\gamma_i^{k+1}-\gamma_i^k|}{\gamma_i^k}\right).
    \end{aligned}
\end{equation}
    As there are $(|\mathcal{K}_{\gamma_i}|-1)$ pairs of consecutive indices for such as $k_1$ and $k_2$, then taking the summation of those inequalities, we have that
    \[
    (|\mathcal{K}_{\gamma_i}|-1)\log(1+\delta)\leq\frac{1}{(1-\eta)}\sum_{k=1}^{\mathcal{N}_\mathrm{iter}-1}\left(\frac{|\gamma_i^{k+1}-\gamma_i^k|}{\gamma_i^k}\right)
    \]
     which completes the proof of \eqref{eqn_K_gamma}. The proofs of \eqref{eqn_K_theta}, \eqref{eqn_K_phi}, and \eqref{eqn_K_psi} are similar.
\end{proof}

Now, it is ready to present the proof of Theorem \ref{theorem_num_rank}. Taking the proof of the bound of $\sum_{k=1}^{\mathcal{N}}I_{\gamma}^k$ as an example, applying the Cauchy–Schwarz inequality for vectors $\left| \frac{\gamma^{k+1}-\gamma^k}{\gamma^k}\right|$ and $e=(1,\cdots,1)\in\rr^{n}$ results in
\begin{equation}
\begin{aligned}
\sum_{i=1}^{n}\left|\frac{\gamma_i^{k+1}-\gamma_i^k}{\gamma_i^k}\right|\cdot 1&\leq\sqrt{\left(\sum_{i=1}^{n} \left| \frac{\gamma_i^{k+1}-\gamma_i^k}{\gamma_i^k}\right|^2\right)}\cdot\sqrt{\sum_{i=1}^{n} 1^2}\\
&=\sqrt{n}\left\|\frac{\gamma^{k+1}-\gamma^k}{\gamma^k}\right\|.
\end{aligned}
\end{equation}
By Lemma \ref{lemma_j_th_number_rank_1_update}, taking the summation of all inequalities \eqref{eqn_K_gamma}, we have that
\begin{equation}
\begin{aligned}
\sum_{i=1}^{n}\left|\mathcal{K}_{\gamma_i} \right| -n &\leq\frac{\sum_{i=1}^{n}\sum_{k=1}^{\mathcal{N}_\mathrm{iter}-1}\left(\frac{|\gamma_i^{k+1}-\gamma_i^k|}{\gamma_i^k}\right)}{\left(1-\eta\right)\log(1+\delta)}   \\
&\leq\frac{\sum_{k=1}^{\mathcal{N}_\mathrm{iter}-1}\sqrt{n}\left\|\frac{\gamma^{k+1}-\gamma^k}{\gamma^k}\right\|}{\left(1-\eta\right)\log(1+\delta)}
\end{aligned}
\end{equation}
Eqn.\  \eqref{eqn_p_q_r_e} in Lemma \ref{lemma_dx_ds_u}} can be written as
\begin{equation}
\begin{aligned}
\sqrt{\sum_{i=1}^{n} \left| \frac{\gamma_i^{k+1}-\gamma_i^k}{\gamma_i^k}\right|^2 + \sum_{i=1}^{n} \left| \frac{\theta_i^{k+1}-\theta_i^k}{\theta_i^k}\right|^2}\leq\eta,\\  
\sqrt{\sum_{i=1}^{n} \left| \frac{\phi_i^{k+1}-\phi_i^k}{\phi_i^k}\right|^2 + \sum_{i=1}^{n} \left| \frac{\psi_i^{k+1}-\psi_i^k}{\psi_i^k}\right|^2}\leq\eta.
\end{aligned}
\end{equation}
That is, we have that
\begin{equation}
\begin{aligned}
    &\left\|\frac{\gamma^{k+1}-\gamma^k}{\gamma^k}\right\|=\sqrt{\sum_{i=1}^{n} \left| \frac{\gamma_i^{k+1}-\gamma_i^k}{\gamma_i^k}\right|^2}\leq\eta\\
    &\left\|\frac{\theta^{k+1}-\theta^k}{\theta^k}\right\|=\sqrt{\sum_{i=1}^{n} \left| \frac{\theta_i^{k+1}-\theta_i^k}{\theta_i^k}\right|^2}\leq\eta\\
    &\left\|\frac{\phi^{k+1}-\phi^k}{\phi^k}\right\|=\sqrt{\sum_{i=1}^{n} \left| \frac{\phi_i^{k+1}-\phi_i^k}{\phi_i^k}\right|^2}\leq\eta\\
    &\left\|\frac{\psi^{k+1}-\psi^k}{\psi^k}\right\|=\sqrt{\sum_{i=1}^{n} \left| \frac{\psi_i^{k+1}-\psi_i^k}{\psi_i^k}\right|^2}\leq\eta
\end{aligned}
\end{equation}
Finally, by \eqref{eqn_I_x}, we have that
\[
\sum_{k=1}^{\mathcal{N}_\mathrm{iter}}\left|I_\gamma^k \right| = \sum_{i=1}^{n}\left|\mathcal{K}_{\gamma_i} \right|-n\leq\frac{\eta(\mathcal{N}_\mathrm{iter}-1)\sqrt{n}}{\left(1-\eta\right)\log(1+\delta)}.
\]
Similarly, we can prove that
\begin{equation}
\begin{aligned}
&\sum_{k=1}^{\mathcal{N}_\mathrm{iter}}\left|I_\theta^k \right|\leq\frac{\eta(\mathcal{N}_\mathrm{iter}-1)\sqrt{n}}{\left(1-\eta\right)\log(1+\delta)},\\
&\sum_{k=1}^{\mathcal{N}_\mathrm{iter}}\left|I_\phi^k \right|\leq\frac{\eta(\mathcal{N}_\mathrm{iter}-1)\sqrt{n}}{\left(1-\eta\right)\log(1+\delta)},\\
&\sum_{k=1}^{\mathcal{N}_\mathrm{iter}}\left|I_\psi^k \right|\leq\frac{\eta(\mathcal{N}_\mathrm{iter}-1)\sqrt{n}}{\left(1-\eta\right)\log(1+\delta)}.
\end{aligned}
\end{equation}
Thus the total number of rank-1 updates is bounded as
\[
\begin{aligned}
    \mathcal{N}_{
    \mathrm{rank}-1}&=\sum_{k=1}^{\mathcal{N}_\mathrm{iter}} |I_{\gamma}^k\cup I_{\theta}^k\cup I_{\phi}^k\cup I_{\psi}^k|\\
    &\leq\sum_{k=1}^{\mathcal{N}_\mathrm{iter}} I_{\gamma}^k +\sum_{k=1}^{\mathcal{N}_\mathrm{iter}} I_{\theta}^k +\sum_{k=1}^{\mathcal{N}_\mathrm{iter}} I_{\phi}^k+\sum_{k=1}^{\mathcal{N}_\mathrm{iter}} I_{\psi}^k\\
    &\leq\frac{4\eta(\mathcal{N}_\mathrm{iter}-1)\sqrt{n}}{\left(1-\eta\right)\log(1+\delta)}\\
    &\leq\left\lceil\frac{4\eta(\mathcal{N}_\mathrm{iter}-1)\sqrt{n}}{\left(1-\eta\right)\log(1+\delta)}\right\rceil,
\end{aligned}
\]
which completes the proof of Theorem \ref{theorem_num_rank}.

\subsection{SVM and SVR to Box-QP}\label{sec_SVM_BoxQP}
Both SVM and SVR can be equivalently transformed into a Box-QP. This subsection takes SVM as an example to illustrate this transformation. 

Given a set of instance-label paris $(x_i,y_i),~i=1,2,\cdots{},m,~x_i\in\mathbb{R}^n,~y_i\in\left\{-1,+1\right\}$, SVM usually first maps training vectors into a high-dimensional feature space via a nonlinear function $\varphi=\varphi(x)$ and then aims to find a hyperplane that separates the points $x_i$ into two groups: one where $y_i=1$, and another where $y_i=-1$:
\begin{equation}
\textrm{find}~w,b~\text{s.t.}~y_i(\varphi_i^\top w+b)\geq1.
\end{equation}
The bias term $b$ is often dealt with by adding an extra dimension to each instance,
\begin{equation}
\varphi_i\leftarrow \mathrm{col}(\varphi_i, 1),\quad  w\leftarrow \mathrm{col}(w, b).
\end{equation}
To address the problems of nonseparable datasets, as well as sensitivity to outliers, the key is to soften the constraints by introducing the slack variables $\zeta$:
\begin{equation}
\textrm{find}~w ~\text{s.t.}~y_i(\varphi_i^\top w)\geq1-\zeta_i,\ \zeta_i\geq0.
\end{equation}
Then, a soft-margin SVM problem is 
\begin{equation}\label{problem_SVM}
    \begin{aligned}
    \min&~\frac{1}{2}w^\top w+ \rho e^\top \zeta\\
    \text{s.t.}&~y_i(\varphi_i^\top w)\geq1-\zeta_i,~\zeta_i\geq0,\ i=1,\cdots{},m.
\end{aligned}
\end{equation}
where $\rho>0$ is a regularization parameter that balances the trade-off between model complexity and training error.
\begin{lemma}
    The optimal solution $w^*$ of soft-margin SVM \eqref{problem_SVM} can be recovered with
    \[
    w^*=\sum_{i=1}^mz_i^*y_i\varphi_i
    \]
    where $z^*$ is the optimal solution of the  Box-QP: \begin{equation}\label{eqn_SVM_BoxQP}
    \begin{aligned}
    \min_{z\in\rr^m}&~\frac{1}{2}\sum_{i=1}^mz_iy_i\varphi_i^\top \sum_{i=1}^mz_iy_i\varphi_i -\sum_{i=1}^mz_i\\
    \mathrm{s.t. }&~0\leq z \leq \rho
    \end{aligned}
    \end{equation}
\end{lemma}
\begin{proof}
According to \cite[Ch 5]{boyd2004convex}, the Karush–Kuhn–Tucker (KKT) condition of soft-margin SVM \eqref{problem_SVM} is
\begin{equation}
\begin{aligned}
    &w - \sum_{i=1}^mz_iy_i\varphi_i=0,\\
    &v + z = \rho e, \\
    &v\pdot \zeta=0,\\
    &v\geq0,~\zeta\geq0,\\
    &z_i \cdot  \left[y_i(\varphi_i^\top w)-1+\zeta_i\right]=0,i=1,\cdots,m,\\
    &z_i\geq0,~y_i(\varphi_i^\top w)-1+\zeta_i\geq0,i=1,\cdots,m,
\end{aligned}
\end{equation}
where $z,v$ denote the Lagrangian dual variables for the inequality constraint $y_i(\varphi_i^\top w)-1+\zeta_i=0,~\zeta_i\geq0~(i=1,\cdots,m)$, respectively. 

By eliminating $w$ with $w=\sum_{i=1}^mz_iy_i\varphi_i$, the KKT conditions of the soft-margin SVM \eqref{problem_SVM} and the Box-QP \eqref{eqn_SVM_BoxQP} are the same, which completes the proof.
\end{proof}
The Box-QP \eqref{eqn_SVM_BoxQP} can be easily scaled to Box-QP \eqref{eqn_box_QP}. The feature space can blow up in size and even be infinite-dimensional. A kernel function $k(x_i,x_j)$ can express a dot product in a (possibly infinite-dimensional) feature space (Mercer’s theorem), that is, $\varphi_i^\top \varphi_j=\varphi(x_i)^\top\varphi(x_j)=k(x_i,x_j)$, to save computations. Therefore, the prediction classifier is $y(x)=\sum_{i=1}^mz_i^*y_ik(x,x_i)$.

\bibliographystyle{IEEEtran}
\bibliography{ref} 

\begin{thebibliography}{10}
\providecommand{\url}[1]{#1}
\csname url@samestyle\endcsname
\providecommand{\newblock}{\relax}
\providecommand{\bibinfo}[2]{#2}
\providecommand{\BIBentrySTDinterwordspacing}{\spaceskip=0pt\relax}
\providecommand{\BIBentryALTinterwordstretchfactor}{4}
\providecommand{\BIBentryALTinterwordspacing}{\spaceskip=\fontdimen2\font plus
\BIBentryALTinterwordstretchfactor\fontdimen3\font minus \fontdimen4\font\relax}
\providecommand{\BIBforeignlanguage}[2]{{%
\expandafter\ifx\csname l@#1\endcsname\relax
\typeout{** WARNING: IEEEtran.bst: No hyphenation pattern has been}%
\typeout{** loaded for the language `#1'. Using the pattern for}%
\typeout{** the default language instead.}%
\else
\language=\csname l@#1\endcsname
\fi
#2}}
\providecommand{\BIBdecl}{\relax}
\BIBdecl

\bibitem{borrelli2017predictive}
F.~Borrelli, A.~Bemporad, and M.~Morari, \emph{Predictive Control for Linear and Hybrid Systems}.\hskip 1em plus 0.5em minus 0.4em\relax U.K.: Cambridge University Press, 2017.

\bibitem{wu2023simple}
L.~Wu and A.~Bemporad, ``A simple and fast coordinate-descent augmented-{Lagrangian} solver for model predictive control,'' \emph{IEEE Transactions on Automatic Control}, vol.~68, no.~11, pp. 6860--6866, 2023.

\bibitem{kuhl2011real}
P.~K{\"u}hl, M.~Diehl, T.~Kraus, J.~P. Schl{\"o}der, and H.~G. Bock, ``A real-time algorithm for moving horizon state and parameter estimation,'' \emph{Computers \& Chemical Engineering}, vol.~35, no.~1, pp. 71--83, 2011.

\bibitem{ames2016control}
A.~D. Ames, X.~Xu, J.~W. Grizzle, and P.~Tabuada, ``Control barrier function based quadratic programs for safety critical systems,'' \emph{IEEE Transactions on Automatic Control}, vol.~62, no.~8, pp. 3861--3876, 2016.

\bibitem{mellinger2011minimum}
D.~Mellinger and V.~Kumar, ``Minimum snap trajectory generation and control for quadrotors,'' in \emph{IEEE International Conference on Robotics and Automation}, 2011, pp. 2520--2525.

\bibitem{escande2014hierarchical}
A.~Escande, N.~Mansard, and P.-B. Wieber, ``Hierarchical quadratic programming: {F}ast online humanoid-robot motion generation,'' \emph{The International Journal of Robotics Research}, vol.~33, no.~7, pp. 1006--1028, 2014.

\bibitem{bouyarmane2018quadratic}
K.~Bouyarmane, K.~Chappellet, J.~Vaillant, and A.~Kheddar, ``Quadratic programming for multirobot and task-space force control,'' \emph{IEEE Transactions on Robotics}, vol.~35, no.~1, pp. 64--77, 2018.

\bibitem{schoeberl2009time}
M.~Schoeberl, ``Time-predictable computer architecture,'' \emph{EURASIP Journal on Embedded Systems}, vol. 2009, pp. 1--17, 2009.

\bibitem{wu2025arbitrarily}
L.~Wu, A.~Adegbege, Y.~Song, and R.~D. Braatz, ``Arbitrarily small execution-time certificate: {W}hat was missed in analog optimization,'' \emph{arXiv preprint arXiv:2505.10366}, 2025.

\bibitem{richter2011computational}
S.~Richter, C.~N. Jones, and M.~Morari, ``Computational complexity certification for real-time {MPC} with input constraints based on the fast gradient method,'' \emph{IEEE Transactions on Automatic Control}, vol.~57, no.~6, pp. 1391--1403, 2011.

\bibitem{giselsson2012execution}
P.~Giselsson, ``Execution time certification for gradient-based optimization in model predictive control,'' in \emph{51st IEEE Conference on Decision and Control}, 2012, pp. 3165--3170.

\bibitem{patrinos2013accelerated}
P.~Patrinos and A.~Bemporad, ``An accelerated dual gradient-projection algorithm for embedded linear model predictive control,'' \emph{IEEE Transactions on Automatic Control}, vol.~59, no.~1, pp. 18--33, 2013.

\bibitem{cimini2017exact}
G.~Cimini and A.~Bemporad, ``Exact complexity certification of active-set methods for quadratic programming,'' \emph{IEEE Transactions on Automatic Control}, vol.~62, no.~12, pp. 6094--6109, 2017.

\bibitem{arnstrom2019exact}
D.~Arnstr{\"o}m and D.~Axehill, ``Exact complexity certification of a standard primal active-set method for quadratic programming,'' in \emph{58th IEEE Conference on Decision and Control}, 2019, pp. 4317--4324.

\bibitem{cimini2019complexity}
G.~Cimini and A.~Bemporad, ``Complexity and convergence certification of a block principal pivoting method for box-constrained quadratic programs,'' \emph{Automatica}, vol. 100, pp. 29--37, 2019.

\bibitem{arnstrom2020complexity}
D.~Arnstr{\"o}m, A.~Bemporad, and D.~Axehill, ``Complexity certification of proximal-point methods for numerically stable quadratic programming,'' \emph{IEEE Control Systems Letters}, vol.~5, no.~4, pp. 1381--1386, 2020.

\bibitem{arnstrom2021unifying}
D.~Arnstr{\"o}m and D.~Axehill, ``A unifying complexity certification framework for active-set methods for convex quadratic programming,'' \emph{IEEE Transactions on Automatic Control}, vol.~67, no.~6, pp. 2758--2770, 2021.

\bibitem{okawa2021linear}
I.~Okawa and K.~Nonaka, ``Linear complementarity model predictive control with limited iterations for box-constrained problems,'' \emph{Automatica}, vol. 125, p. 109429, 2021.

\bibitem{wu2023direct}
L.~Wu and R.~D. Braatz, ``A direct optimization algorithm for input-constrained {MPC},'' \emph{IEEE Transactions on Automatic Control}, vol.~70, no.~2, pp. 1366--1373, 2025.

\bibitem{wu2024time}
L.~Wu, K.~Ganko, and R.~D. Braatz, ``Time-certified input-constrained {NMPC} via {K}oopman operator,'' in \emph{8th IFAC Conference on Nonlinear Model Predictive Control}, 2024, in press, arXiv:2401.04653.

\bibitem{wu2024execution}
L.~Wu, K.~Ganko, S.~Wang, and R.~D. Braatz, ``An execution-time-certified {R}iccati-based {IPM} algorithm for {RTI}-based input-constrained {NMPC},'' in \emph{63nd IEEE Conference on Decision and Control}, 2024, pp. 5539--5545.

\bibitem{wu2024parallel}
L.~Wu, L.~Zhou, and R.~D. Braatz, ``A parallel vector-form {$LDL^\top$} decomposition for accelerating execution-time-certified $\ell_1$-penalty soft-constrained {MPC},'' \emph{arXiv preprint arXiv:2403.18235}, 2024.

\bibitem{roos2006full}
C.~Roos, ``A full-{N}ewton step {$O(n)$} infeasible interior-point algorithm for linear optimization,'' \emph{SIAM Journal on Optimization}, vol.~16, no.~4, pp. 1110--1136, 2006.

\bibitem{wu2025eiqp}
L.~Wu, W.~Xiao, and R.~D. Braatz, ``E{IQP}: {E}xecution-time-certified and infeasibility-detecting {QP} solver,'' \emph{arXiv preprint arXiv:2502.07738}, 2025.

\bibitem{karmarkar1984new}
N.~Karmarkar, ``A new polynomial-time algorithm for linear programming,'' in \emph{Proceedings of the Sixteenth Annual ACM Symposium on Theory of Computing}, 1984, pp. 302--311.

\bibitem{vaidya1987algorithm}
P.~M. Vaidya, ``An algorithm for linear programming which requires ${O}(((m+n)n^2+(m+n)^{1.5}n){L})$ arithmetic operations,'' in \emph{Proceedings of the Nineteenth Annual ACM symposium on Theory of Computing}, 1987, pp. 29--38.

\bibitem{gonzaga1989algorithm}
C.~C. Gonzaga, ``An algorithm for solving linear programming problems in {$O(n^3L)$} operations,'' in \emph{Progress in Mathematical Programming: Interior-Point and Related Methods}, N.~Megiddo, Ed.\hskip 1em plus 0.5em minus 0.4em\relax New York: Springer, 1989, pp. 1--28.

\bibitem{roos1989}
C.~Roos, ``An ${0}(n^3{L})$ approximate center method for linear programming,'' in \emph{Optimization: Proceedings of the Fifth French-German Conference held in Castel-Novel (Varetz), France, Oct. 3--8, 1988}.\hskip 1em plus 0.5em minus 0.4em\relax Springer, 1989, pp. 147--158.

\bibitem{ye1991n}
Y.~Ye, ``An ${O}(n^3{L})$ potential reduction algorithm for linear programming,'' \emph{Mathematical Programming}, vol.~50, no.~1, pp. 239--258, 1991.

\bibitem{mizuno1991nrhol}
S.~Mizuno, ``${O}(n^{\rho}{L}$)-iteration and ${O}(n^3{L})$-operation potential reduction algorithms for linear programming,'' \emph{Linear Algebra and Its Applications}, vol. 152, pp. 155--168, 1991.

\bibitem{anstreicher1992long}
K.~M. Anstreicher and R.~A. Bosch, ``Long steps in an ${O}(n^3{L})$ algorithm for linear programming,'' \emph{Mathematical Programming}, vol.~54, pp. 251--265, 1992.

\bibitem{bosch1993mizuno}
R.~A. Bosch, ``On {M}izuno’s rank-one updating algorithm for linear programming,'' \emph{SIAM Journal on Optimization}, vol.~3, no.~4, pp. 861--867, 1993.

\bibitem{wu1994rank}
S.~Wu and F.~Wu, ``Rank-one techniques in log-barrier function methods for linear programming,'' \emph{Journal of optimization theory and applications}, vol.~82, no.~2, pp. 405--413, 1994.

\bibitem{anstreicher1999linear}
K.~M. Anstreicher, ``Linear programming in ${O}([n^3/\ln n]{L})$ operations,'' \emph{SIAM Journal on Optimization}, vol.~9, no.~4, pp. 803--812, 1999.

\bibitem{monteiro1989interiorII}
R.~D. Monteiro and I.~Adler, ``Interior path following primal-dual algorithms. {P}art {II}: {C}onvex quadratic programming,'' \emph{Mathematical Programming}, vol.~44, no.~1, pp. 43--66, 1989.

\bibitem{goldfarb1990n}
D.~Goldfarb and S.~Liu, ``An ${O}(n^3{L})$ primal interior point algorithm for convex quadratic programming,'' \emph{Mathematical programming}, vol.~49, no.~1, pp. 325--340, 1990.

\bibitem{tiande1996properties}
G.~Tiande and W.~Shiquan, ``Properties of primal interior point methods for {QP},'' \emph{Optimization}, vol.~37, no.~3, pp. 227--238, 1996.

\bibitem{kojima1989polynomial}
M.~Kojima, S.~Mizuno, and A.~Yoshise, ``A polynomial-time algorithm for a class of linear complementarity problems,'' \emph{Mathematical programming}, vol.~44, pp. 1--26, 1989.

\bibitem{mizuno1990n}
S.~Mizuno, ``An ${O}(n^{3}{L})$ algorithm using a sequence for a linear complementarity problem,'' \emph{Journal of the Operations Research Society of Japan}, vol.~33, no.~1, pp. 66--75, 1990.

\bibitem{mizuno1991n}
S.~Mizuno and M.~J. Todd, ``An ${O}(n^3{L})$ adaptive path following algorithm for a linear complementarity problem,'' \emph{Mathematical Programming}, vol.~52, no.~1, pp. 587--595, 1991.

\bibitem{mizuno1992new}
S.~Mizuno, ``A new polynomial time method for a linear complementarity problem,'' \emph{Mathematical Programming}, vol.~56, no.~1, pp. 31--43, 1992.

\bibitem{mcshane1996superlinear}
K.~A. McShane, ``On the superlinear convergence of an ${O}(n^3{L})$ interior-point algorithm for the monotone {LCP},'' \emph{SIAM Journal on Optimization}, vol.~6, no.~4, pp. 978--993, 1996.

\bibitem{tibshirani1996regression}
R.~Tibshirani, ``Regression shrinkage and selection via the lasso,'' \emph{Journal of the Royal Statistical Society Series B: Statistical Methodology}, vol.~58, no.~1, pp. 267--288, 1996.

\bibitem{cortes1995support}
C.~Cortes and V.~Vapnik, ``Support-vector networks,'' \emph{Machine learning}, vol.~20, pp. 273--297, 1995.

\bibitem{smola2004tutorial}
A.~J. Smola and B.~Sch{\"o}lkopf, ``A tutorial on support vector regression,'' \emph{Statistics and Computing}, vol.~14, pp. 199--222, 2004.

\bibitem{boyd2004convex}
S.~P. Boyd and L.~Vandenberghe, \emph{Convex optimization}.\hskip 1em plus 0.5em minus 0.4em\relax Cambridge University Press, 2004.

\bibitem{wright1997primal}
S.~J. Wright, \emph{Primal-dual interior-point methods}.\hskip 1em plus 0.5em minus 0.4em\relax SIAM, 1997.

\bibitem{bemporad1997nonlinear}
A.~Bemporad, A.~Casavola, and E.~Mosca, ``Nonlinear control of constrained linear systems via predictive reference management,'' \emph{IEEE transactions on Automatic Control}, vol.~42, no.~3, pp. 340--349, 1997.

\bibitem{mattingley2010real}
J.~Mattingley and S.~Boyd, ``Real-time convex optimization in signal processing,'' \emph{IEEE Signal Processing Magazine}, vol.~27, no.~3, pp. 50--61, 2010.

\bibitem{forces_low_level_estimation}
{Embotech AG}, ``Low-{L}evel {R}obust {E}stimation {E}xample --- {FORCES} {P}ro {D}ocumentation,'' \url{https://forces.embotech.com/Documentation/examples/low_level_robust_estimation/index.html}.

\end{thebibliography}

\begin{IEEEbiography}[{\includegraphics[width=1in,height=1.25in,clip,keepaspectratio]{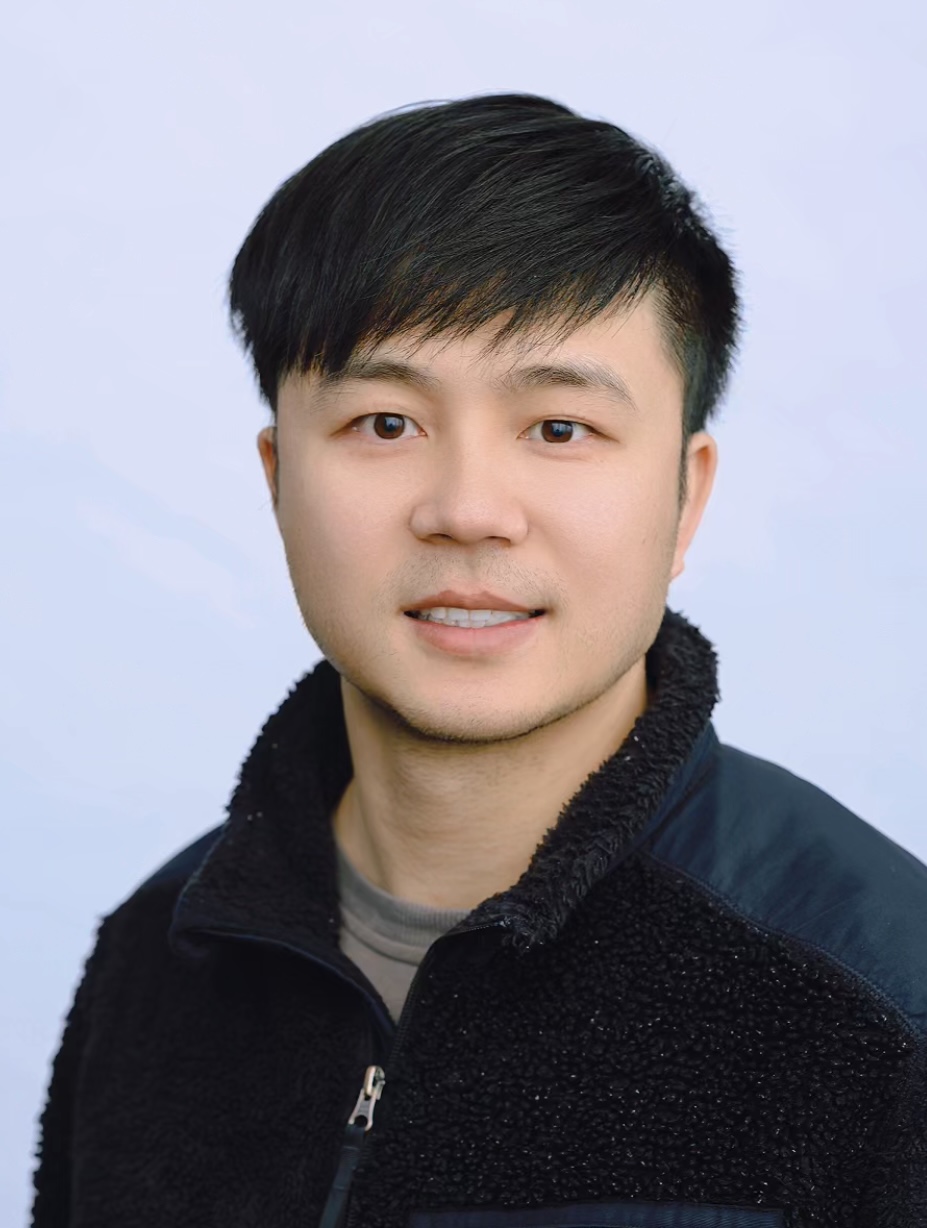}}]{Liang Wu}
received a B.E. and M.E. in chemical engineering, majoring in computational fluid dynamics and process modeling. He then turned his research interest to model predictive control and received his Ph.D. in March 2023 from the IMT School for Advanced Studies Lucca supervised by Prof.\ Alberto Bemporad. 

He is now a postdoctoral associate at the Massachusetts Institute of Technology (MIT) supervised by Prof.\  Richard D. Braatz since March 27, 2023. His research interests include quadratic programming, model predictive control, convex nonlinear programming, and control of PDE systems.
\end{IEEEbiography}

\begin{IEEEbiography}[{\includegraphics[width=1in,height=1.25in,clip,keepaspectratio]{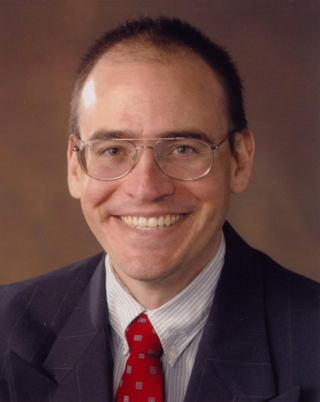}}]{Richard D. Braatz}
is the Edwin R. Gilliland Professor at the Massachusetts Institute of Technology (MIT). He received an M.S. and
Ph.D. from the California Institute of Technology. He is a past Editor-in-Chief of the IEEE Control Systems Magazine, Past President of the American Automatic Control Council, a Vice President of the International Federation of Automatic Control, and was General Co-Chair of the 2020 IEEE Conference on Decision and Control. He received the AACC Donald P. Eckman Award, the Antonio Ruberti Young Researcher
Prize, and best paper awards from IEEE- and IFAC-sponsored control journals. He is a Fellow of IEEE and IFAC and a member
of the U.S. National Academy of Engineering.
\end{IEEEbiography}

\end{document}